\documentclass[letterpaper,11pt]{amsart}

\usepackage[all]{xy}                        %
\usepackage{color}
\CompileMatrices                            

\UseTips                                    

\input xypic
\usepackage[bookmarks=true]{hyperref}       

\usepackage{amssymb,latexsym,amsmath,amscd}
\usepackage{xspace}

\usepackage{graphicx}

\reversemarginpar

\theoremstyle{plain}
\newtheorem{theorem}{Theorem}[section]
\newtheorem*{theorem*}{Theorem}
\newtheorem{proposition}[theorem]{Proposition}

\newtheorem{lemma}[theorem]{Lemma}

\theoremstyle{definition}
\newtheorem{definition}[theorem]{Definition}

\newtheorem{remark}[theorem]{Remark}
\newtheorem{example}[theorem]{Example}

\newcommand{\enm}[1]{\ensuremath{#1}}          %
\newcommand{\op}[1]{\operatorname{#1}}
\newcommand{\cal}[1]{\mathcal{#1}}

\newcommand{\NN}{\enm{\mathbb{N}}}

\newcommand{\ZZ}{\enm{\mathbb{Z}}}

\renewcommand{\AA}{\enm{\mathbb{A}}}
\newcommand{\PP}{\enm{\mathbb{P}}}

\newcommand{\Aa}{\enm{\cal{A}}}

\newcommand{\Ee}{\enm{\cal{E}}}
\newcommand{\Ff}{\enm{\cal{F}}}

\newcommand{\Ii}{\enm{\cal{I}}}

\newcommand{\Ll}{\enm{\cal{L}}}

\newcommand{\Nn}{\enm{\cal{N}}}
\newcommand{\Oo}{\enm{\cal{O}}}

\newcommand{\Rr}{\enm{\cal{R}}}

\renewcommand{\phi}{\varphi}
\renewcommand{\theta}{\vartheta}
\renewcommand{\epsilon}{\varepsilon}

\newcommand{\Pic}{\op{Pic}}

\newcommand{\Ext}{\op{Ext}}
\newcommand{\Sec}{\op{Sec}}

\renewcommand{\to}[1][]{\xrightarrow{\ #1\ }}

\newcommand{\old}[1]{}

\begin{document}
\title[Globally generated vector bundles]{Globally generated vector bundles on complete intersection Calabi-Yau threefolds}
\author{E. Ballico, S. Huh and F. Malaspina}
\address{Universit\`a di Trento, 38123 Povo (TN), Italy}
\email{edoardo.ballico@unitn.it}
\address{Sungkyunkwan University, Suwon 440-746, Korea}
\email{sukmoonh@skku.edu}
\address{Politecnico di Torino, Corso Duca degli Abruzzi 24, 10129 Torino, Italy}
\email{francesco.malaspina@polito.it}
\keywords{Complete intersection Calabi-Yau, Vector bundles, Globally generated, Curves in projective spaces}
\thanks{The first and third authors are partially supported by MIUR and GNSAGA of INDAM (Italy). The second author is supported by Basic Science Research Program 2010-0009195 through NRF funded by MEST. The third author is supported by the framework of PRIN 2010/11 \lq Geometria delle variet\`a algebriche\rq, cofinanced by MIUR}
\subjclass[2010]{14J60; 14J32; 14H25}

\begin{abstract}
We investigate the globally generated vector bundles on complete intersection Calabi-Yau threefolds with the first Chern class at most $2$. We classify all the globally generated vector bundles of an arbitrary rank on quintic in $\PP^4$ and investigate the globally generated vector bundles of rank $2$ on complete intersection Calabi-Yau threefolds of codimension $2$.
\end{abstract}

\maketitle

\section{Introduction}
Globally generated vector bundles on projective varieties play an important role in classical algebraic geometry. In fact every globally generated bundles $\Ee$ of rank $k$ on an algebraic variety $X$ and every $(N+1)$-dimensional linear subspace $V\subset H^0(E)$ such that $V\otimes \Oo_X\to\Ee$ is an epimorphism, determine a regular morphism
$$\phi_V: X \rightarrow \mathbb{G}(k, N+1)$$
to the Grassmannian variety parametrizing $k$-dimensional subspaces in the dual space of $V$. Conversely every such a regular morphism corresponds to a globally generated vector bundle $\Ee$ of rank $k$ on $X$ and an $(N+1)$-dimensional linear subspace $V\subset H^0(\Ee)$ (see \cite[Section $3$]{sierra}).

If globally generated vector bundles are nontrivial, then they must have strictly positive first Chern class. The classification of globally generated vector bundles with low first Chern class has been done over several rational varieties such as projective spaces \cite{am,SU} and quadric hypersurfaces \cite{BHM}.

In this paper we examine the similar problem for the complete intersection Calabi-Yau (CICY for short) threefold. There are exactly five types of such threefolds: the quintic in $\PP^4$, the complete intersections $(3,3)$ and $(2,4)$ in $\PP^5$, the complete intersection $(2,2,3)$ in $\PP^6$ and finally $(2,2,2,2)$ in $\PP^7$. Calabi-Yau threefolds in general have been objects of extended interest in algebraic geometry, mainly because of its relations with the theory of mirror symmetry. Especially showing the existence and counting the number of projective curves in Calabi-Yau threefolds have been the main problems in many works, for example see \cite{ES,kanazawa,Knutsen}.

The Hartshorne-Serre correspondence states that the construction of vector bundles of rank $r\ge 2$ on a smooth variety $X$ with dimension $ 3$ is closely related with the structure of curves in $X$ and it inspires the classification of vector bundles on Calabi-Yau threefolds. There have been several works on the classification of {\it arithmetically Cohen-Macaulay} (ACM) bundles on Calabi-Yau threefolds, e.g. \cite{madonna} and so it is sufficiently timely to classify the globally generated vector bundles on Calabi-Yau threefolds. On the other hand, several works on mirror symmetry for Calabi-Yau complete intersections in Grassmannians, e.g. \cite{BCKS, HK}, keep giving us inspiration during this work, mainly because such bundles may enable us to consider CICY as subvarieties of Grassmannians.

Our first main result is the following:

\begin{theorem}\label{theorem1.1}
Let $\Ee$ be a globally generated vector bundle of rank $r\ge 2$ without trivial factors on a smooth quintic hypersurface $X=X_5$ in $\PP^4$.
If $c_1(\Ee)\le 2$, then one of the following holds:\\
$($$\pi_P: X \to \PP^3$ is a linear projection from a point $P\in \PP^4\setminus X$.$)$
\begin{enumerate}
\item $\Ee \cong T\PP^4(-1)_{\vert_{X}}$ or $\Ee\cong\pi_P^* (T_{\PP^3}(-1))$.
\item $\Ee\cong\Oo_{X}(1)\oplus\Oo_X(1)$, or $\Ee\cong\pi_P^* (\Nn_{\PP^3}(1))$ where $\Nn_{\PP^3}$ is a null correlation bundle on $\PP^3$.
\item $\Ee\cong\pi_P^* (\Omega_{\PP^3}(2))$.
\item $0 \to \Oo_X(-2) \to \Oo_X^{\oplus (r+1)} \to \Ee \to 0$, with $3\le r\le 14$.
\item $0\to \Oo_X(-1)^{\oplus 2} \to \Oo_X^{\oplus (r+2)} \to \Ee \to 0$, with $3\le r\le 8$.
\item $0\to \Oo_X(-1) \to \Oo_X^{\oplus r} \oplus \Oo_X(1) \to \Ee \to 0$, with $3\le r\le 5$.
\end{enumerate}
\end{theorem}

In particular we have $c_2(\Ee) \in \{0,5,10,15, 20\}$ and all are the multiples of $5=\deg (X)$, less than or equal to $c_1(\Ee)^2\times \deg (X)=20$.

The second main result is on CICY threefolds of codimension $2$.
\begin{theorem}\label{theorem1.2}
Let $\Ee$ be a globally generated vector bundle of rank $2$ on a CICY threefold $X$ of codimension $2$. If $c_1(\Ee)\le 2$, then the possible $c_2(\Ee)$ is as follows:
\begin{enumerate}
\item On $X=X_{2,4}$, we have $c_2(\Ee)\in \{ 0,4, 8, 11, 16\}$.
\item On $X=X_{3,3}$, we have $c_2(\Ee) \in \{0,9,12, 15, 16,18\}$.
\end{enumerate}
And for each Chern classes there exist corresponding globally generated vector bundles on $X$.
\end{theorem}
Indeed we describe completely the corresponding curves in $X$ by the Hartshorne-Serre construction, except the case of $c_2(\Ee)=16$, where we show the existence but we did not succeed in the full classification of vector bundles.

The paper is divided into four parts in a natural way:

In the first part, Section $2$, we collect some necessary technical information about CICY threefolds, the Hartshorne-Serre construction and the Castelnuovo's bound for curves in projective spaces. Then we propose a general result on the bound of degrees of the associated curves to globally generated vector bundles.

The second part, mainly Section $3$, is devoted to the proof of Theorem \ref{theorem1.1}. We first classify the globally generated vector bundle on $X_5$ with $c_1(\Ee)=1$ and observe that every such bundles are obtained as pullbacks of bundles on $\PP^3$ by a linear projection or restriction of bundles on $\PP^4$. We investigate the globally generated vector bundles on $X_5$ with $c_1(\Ee)=2$ in terms of the locally free resolution and observe that they have the same list in \cite[Theorem1.1]{SU}.

The third part consisting of Sections $4$-$6$, is devoted to the proof of Theorem \ref{theorem1.2}. In Section $4$ we classify globally generated vector bundle on CICY threefolds $X$ with $c_1(\Ee)=1$, or $c_1(\Ee)=2$ and $h^0(\Ee(-1))>0$. Then we suggest four main types of vector bundles as examples and check the existence. In Section $5$ and $6$ are separately devoted to proving that these four types of vector bundles are only possibilities of globally generated vector bundles on $X$ with $c_1(\Ee)=2$ and $h^0(\Ee(-1))=0$. We investigate the scheme-theoretic base locus of the quadric hypersurfaces of $\PP^5$ containing the associated curve $C$ to a vector bundle and check the existence one by one with respect to dimension and the degree of the irreducible components of the reduced scheme of the base locus.

In the last part, mainly Section $7$, we construct an example of globally generated vector bundles of rank $2$ on CICY threefolds of codimension $3$ with $c_1(\Ee)=2$, mainly using the techniques in the previous sections. We do not know other examples, while our tools are not enough to say that they are the only possibilities.

We would like to thank Carlo Madonna for introducing the problem to us.


\section{Preliminary}
\begin{definition}
A smooth $3$-dimensional projective variety $X$ is called a {\it Calabi-Yau} threefold if its canonical class is trivial. In particular, if a complete intersection $X=X_{r_1, \ldots, r_k} \subset \PP^{k+3}$ of hypersurfaces of degree $r_1, \ldots, r_k$, is Calabi-Yau, then it is called a {\it complete intersection Calabi-Yau} (CICY for short).
\end{definition}

\begin{remark}
It turns out that there are only five types of CICY threefolds:
\begin{enumerate}
\item the quintic $X_5\subset \PP^4$,
\item the intersection $X_{2,4}\subset \PP^5$ of a quadric and a quartic,
\item the intersection $X_{3,3}\subset \PP^5$ of two cubics,
\item the intersection $X_{2,2,3}\subset \PP^6$ of two quadrics and a cubic, and
\item the intersection $X_{2,2,2,2}\subset \PP^7$ of four quadrics.
\end{enumerate}
Thus we can denote each CICY by $X_{u}$ with $u=r_1\cdots r_k$ instead of $X_{r_1, \ldots, r_k}$, e.g. a CICY threefold $X_8$ means the complete intersection $X_{2,4}$.
\end{remark}


\begin{remark}
For a projective variety $X\subset \PP^n$ with an ample line bundle $H=\Oo_X(1)$, consider the exact sequence
$$0\to \Oo_X(-c_1) \to \Oo_X^{\oplus (r+1)} \to \Ff \to 0$$
defined by $r+1$ general elements of $H^0(\Oo_X(c_1))$, where $3\leq r \leq h^0(\Oo_X(c_1))-1$ and $c_1 \in \NN$. Note that any $4$ general elements in $|\Oo_X(c_1)|$ have no common zero and so $\Ff$ is locally free and globally generated. All globally generated vector bundles with such Chern classes are obtained in this way; e.g. in the case of $c_1=1$ and $r=h^0(\Oo_X(1))-1$, we have $\Ff \cong T\PP^n (-1)_{\vert_X}$.
\end{remark}

\begin{remark}
Let $X=X_{2,4}=U_2\cap U_4$ be the intersection of a quadric $U_2$ and a quartic $U_4$. Note that $U_2$ is uniquely determined by $X$, i.e. $U_2$ is the unique quadric hypersurface in $\PP^5$ containing $X$. In general $U_2$ is smooth, but there are some $X$ for which $U_2$ has an isolated singular point, while $U_4$ does not contain that singular point.
\end{remark}

Let $\Ee$ be a vector bundle of rank $2$ on a CICY threefold $X_u$ and then we have $c_1(\Ee(t))=c_1(\Ee)+2t$ and $c_2(\Ee(t))=c_2(\Ee)+utc_1(\Ee)+ut^2$. The Grothendieck-Riemann-Roch formula also gives
\begin{equation}\label{GRR}
\chi(\Ee)=\frac{u}{6}c_1^3-\frac{c_1c_2}{2}+ \frac{c_1}{12} (12(v+4)-2u), \text{ where } v=\lfloor \frac{u}{4} \rfloor.
\end{equation}

\begin{proposition}\cite{sierra}\label{prop1}
Let $X$ be a smooth projective variety with $\Pic (X)=\ZZ $ generated by $H=\Oo_X(1)$. If $\Ee$ is a globally generated vector bundle of rank $r$ on $X$ such that $H^0(\Ee(-c_1))\not= 0$, where $c_1$ is the first Chern class of $\Ee$, then we have $\Ee\simeq \Oo_X^{\oplus (r-1)}\oplus \Oo_X(c_1)$.
\end{proposition}
In particular, $\Ee \simeq \Oo_X^{\oplus r}$ is the unique globally generated vector bundle of rank $r$ on $X$ with $c_1=0$. Thus we can assume that $c_1\geq 1$.

Let $X$ be a smooth Calabi-Yau threefold, not necessarily with $\Pic (X)\cong \mathbb {Z}$.
In the study of globally generated vector bundles on $X$ it is natural to try to look first at low $c_1$ vector bundles. Since globally generated vector bundles with trivial determinant are trivial, it is natural to look first at the bundles with the least positive line bundle as determinant on $X$. When $\Pic (X)\ne \mathbb {Z}$, we find it easy to do the following first step.

We say that a line bundle $\Rr$ on $X$ has property $\diamond$ if the following conditions are satisfied:
\begin{enumerate}
\item $\Rr$ is globally generated;
\item $\Rr \ne \Oo _X$;
\item If $\Aa$ is a globally generated line bundle with $H^0(\Rr\otimes \Aa^\vee )\ne 0$,
then we have either $\Aa \cong \Oo _X$ or $\Aa\cong \Rr$.
\end{enumerate}

\begin{remark}\label{a01}
Let $\Rr$ be a line bundle with $\diamond$. Then any $D\in |\Rr|$ is non-empty, reduced and irreducible. The Bertini theorem gives that $D$ is smooth.
\end{remark}

\begin{lemma}\label{a0}
Let $\Rr$ be a line bundle with $\diamond$. Then we have $h^1(\Rr^\vee )=0$.
\end{lemma}

\begin{proof}
Since $X$ is Calabi-Yau, then we have $h^1(\Oo _X)=0$. For a fixed $D\in |\Rr|$, we have $D \ne \emptyset$ and $D$
is reduced and irreducible by the property $\diamond$ on $\Rr$. In particular we have $h^0(\Oo _D)=1$. Using the exact sequence$$0\to \Rr^\vee \to \Oo _X \to \Oo _D\to 0$$we have $h^1(\Rr^\vee)=0$.
\end{proof}

Now assume that $\Ee$ is globally generated vector bundle of rank $r$ on $X$ with $\det (\Ee)=\Rr$ having the propoerty $\diamond$ and then it fits into the following exact sequence by \cite[Section $2$. $\mathbf{G}$]{man}
\begin{equation}\label{eqn1}
0\to \Oo_X^{\oplus (r-1)} \to \Ee \to \Ii_C\otimes \Rr  \to 0,
\end{equation}
where $C$ is a smooth curve of degree $c_2(\Ee)$ on $X$. If $C$ is empty, then $\Ee$ is isomorphic to $\Oo_X^{\oplus (r-1)}\oplus \Rr$ and so we may assume that $C$ is not empty. It would mean that we assume that $H^0(\Ee\otimes \Rr^\vee)=0$. Restricting the sequence (\ref{eqn1}) to $C$, we have
$$0\to \omega_C^\vee \otimes \Rr \to \Oo_C^{\oplus (r-1)} \to \Ee_{\vert_C} \to N_{C|X}^\vee \otimes \Rr \to 0$$
and in particular $\omega_C \otimes \Rr^\vee$ is globally generated. Lemma \ref{a0} and the Hartshorne-Serre correspondence give the converse of this argument:

\begin{theorem}\cite[Theorem 1]{Arrondo}
Let $\Rr$ be a line bundle on $X$ with $\diamond$.
\begin{enumerate}
\item[(i)] There is a bijection between the set of pairs $(\Ee ,j)$, where $\Ee$ is a spanned vector bundle of rank $2$ on $X$ with $\det (\Ee )\cong \Rr$ and $j: \mathbb{K} \to H^0(\Ee)$ is non-zero map, up to linear automorphisms of $\mathbb{K}$, and the smooth curves $C\subset X$ with $\Ii _C \otimes \Rr$ spanned and $\omega _C \cong \Rr_{\vert _C}$, except that $C =\emptyset$ corresponds to $\Oo _X\oplus \Rr$ with a section $s$ nowhere vanishing.

\item[(ii)] Fix an integer $r\ge 3$. There is a bijection between the set of triples $(\Ee ,V,j)$, where $\Ee$ is a spanned vector bundle of rank $r$ on $X$ with $\det (\Ee )\cong \Rr$, $V$ is an $(r-1)$-dimensional vector space and $j: V \to H^0(\Ee)$ is a linear map, up to a linear automorphism of $V$, with dependency locus of codimension $2$, and the smooth curves $C\subset X$ with $\Ii _C \otimes \Rr$ spanned and $\omega _C \otimes \Oo_X\otimes \Rr^\vee$ spanned, except that $C =\emptyset$ corresponds to $\Oo _X^{\oplus (r-1)}\oplus \Rr$ with $V = H^0(\Oo _X^{\oplus (r-1)})$. $\Ee$ has no trivial factors if and only if $j$ is injective.

\item[(iii)] There is a bijection between the set of all pairs $(\Ff ,s)$, where $\Ff$ is a spanned reflexive sheaf of rank $2$ on $X$ with $\det (\Ee )\cong \Rr$ and $0\ne s\in H^0(\Ee )$, and reduced curves $C\subset X$ with $\Ii _C\otimes \Rr$ spanned and $\omega _C \otimes \Rr^\vee$ spanned outside finitely many points, except that $C =\emptyset$ corresponds to $\Oo _X\oplus \Rr$ with $s$ nowhere vanishing.
\end{enumerate}
\end{theorem}
\begin{proof}
For part (3), we can generalize the theory to any reduced curve $C$
\end{proof}

It enables us to classify the globally generated vector bundles via the classification of curves $C$ in $\PP^n$ with proper numeric invariants such that $\omega_C(-c_1)$ is spanned. Write $\pi (d,n)$ for the upper bound on the genus for non-degenerate curves of degree $d$ in $\PP^n$. Recall that we have in \cite[Theorem 3.7]{he} its computations.

In general let $X$ be a smooth and connected projective threefold with $\Pic (X) \cong \mathbb {Z}$, generated by an ample line bundle $H=\mathcal {O}_X(1)$. Set $\Oo _X(e):= \omega _X$. Assume moreover that the irregularity of $X$ is $0$, e.g. $X$ is a CICY. If $\Ee$ is a globally generated vector bundle of rank $2$ on $X$ with $c_1(\Ee)=2$,
$$0\to \Oo_X \to \Ee \to \Ii_C(2) \to 0,$$
where $C$ is a smooth curve with $\omega_C \cong \Oo_C(2)$.

\begin{lemma}
We have $\deg (C) \leq 4 \deg (X)-3$.
\end{lemma}
\begin{proof}
Since $\Ii_C(2)$ is spanned, so $C$ is contained in the complete intersection of quadric hypersurfaces of $X$, and in particular we have
$$\deg (C)=p_a(C)-1\leq 4\deg (X).$$
The equality holds if and only if $C$ is the complete intersection of two hypersurfaces in $|\Oo_X(2)|$, but it is impossible because it would give $\omega_C \cong \Oo_C(4)$ by the adjunction formula.

Let us assume $\deg (C) = 4\deg (X)-1$ and take two general $Y_1,Y_2\in |\Ii _C(2)|$. Then we get $Y_1\cap Y_2 = C\cup D$ with $D$ a line and the adjunction formula gives $\omega _{Y_1\cap Y_2} \cong \Oo _{Y_1\cap Y_2}(4)$. So we have $\deg (C\cap D) =4\deg (X) -\deg (\omega _D)\ge6$ and it implies that each quadric hypersurface containing $C$ contains $D$ and so $h^0(\Ii _C(2)) =2$. In particular $\Ii _C(2)$.

Now let us assume $\deg (C)  =4\deg (X)-2$ and again take two general $Y_1,Y_2\in |\Ii _C(2)|$. We get $Y_1\cap Y_2 = C\cup D$ with $D$ either two disjoint lines or a reduced conic or a double structure on a line. First assume that $D$ is the disjoint union of two lines, say $D =D_1\sqcup D_2$. Since $\omega _{Y_1\cap Y_2} \cong \Oo _{Y_1\cap Y_2}(4)$, we get $\deg (D_i\cap C) =6$ and so $h^0(\Ii _C(2)) =2$. In particular $\Ii _C(2)$ is not spanned. If $D$ is a reduced conic, we
get $\deg (D\cap C) =6$ and again $\Ii _C(2)$ is not spanned; if $D$ is reducible, then at least one of its components, say $T$, satisfies $\deg (T\cap C) >2$. The case of $D$ with a double structure on a line does not arise for a general $Y_1,Y_2$, because $Y_1\cap Y_2$ would be smooth outside $C$ by the Bertini theorem.
\end{proof}


\section{Quintic Hypersurface}
In this section let us assume that $X=X_5$ is a smooth quintic hypersurface in $\PP^5$.

\begin{example}\label{c1}
Fix two planes $U_1, U_2\subset \PP^4$ such that $U_1\cup U_2$ spans $\PP^4$, i.e. $U_1\cap U_2$ is a single point $P$ and asume $P\notin X$ and $C_i:= X\cap U_i$, $i=1,2$, is a smooth plane curve of degree $5$. Setting $C = C_1\cup C_2$ and $U := U_1\cup U_2$, we have $\omega _C \cong \Oo _C(2)$ and $h^0(\Ii _C(1)) =0$. Therefore to get the globally generated bundle $\Ee$ associated to $C$ it is sufficient to prove that $\Ii _C(2)$ is globally generated. Since $C$ is the complete intersection of $X$ and $U$, it is sufficient to prove that $\Ii _{U,\PP^4}(2)$ is globally generated at all points of $X$. Indeed its global sections span $\Ii _{U,\PP ^4}(2)$ outside $P$: Fix a hyperplane $H\subset  \PP ^4$ with $P\notin H$. Since $U$ is a cone with vertex $P$, it is sufficient to check that $\Ii _{U\cap H,H}(2)$ is globally generated, which is true because $U\cap H$ is the union of two disjoint lines.
\end{example}

The example \ref{c1} gives us a globally generated vector bundle of rank $2$ on $X$ with the Chern classes $(c_1, c_2)=(2,10)$.

\begin{remark}\label{remc1}
Let $\pi_P: X_5 \to \PP^3$ be a linear projection from a point $P\in \PP^4 \setminus X_5$, which is the intersection $U_1\cap U_2$. Then the image of $C:=C_1 \cup C_2$ under $\pi$ is $D:=L_1 \cap L_2$, the union of two skew lines in $\PP^3$, where $L_i:=\pi_P (C_i)$. In $\PP^3$ we have the extension
$$0\to \Oo_{\PP^3} \to \Ff \to \Ii_D(2) \to 0$$
and we have $\Ff \cong \Nn_{\PP^3}(1)$, a null correlation bundle on $\PP^3$ twisted by $1$. In view of \cite[Theorem 1.1 and Remark 1.1.1]{hs}, the bundle in Example \ref{c1} is isomorphic to $\pi_P^* (\Nn_{\PP^3}(1))$.
\end{remark}

\subsection{Case of rank $2$} $ $

Our goal in this subsection is to prove the following theorem.

\begin{theorem}\label{thm1}
Let $\Ee$ be an globally generated vector bundle of rank $2$ on $X=X_5\subset \PP^4$ with the Chern classes $(c_1, c_2)$ and $c_1 \leq 2$. Then we have either
\begin{enumerate}
\item $\Ee \cong \Oo_X(a)\oplus \Oo_X(b)$ with $(a,b)\in \{(0,0),(0,1), (0,2), (1,1)\}$ or
\item $\Ee\cong \pi_P^* (\Nn_{\PP^3}(1))$, where $\pi_P : X_5 \to \PP^3$ is a linear projection from a point $P\in \PP^4 \setminus X_5$.
\end{enumerate}
In particular, we have $(c_1, c_2)\in \{(1,0), (2,0), (2,5), (2,10)\}$.
\end{theorem}
\begin{proof}
By Proposition \ref{prop1} we can assume that $H^0(\Ee(-c_1))=0$. If $c_1(\Ee)=1$, then $\Ee$ fits into the sequence
$$0\to \Oo_X \to \Ee \to \Ii_C(1) \to 0, $$
where $C$ is a smooth curve. If $h^0(\Ii_C(1))\geq 3$, then we have $C \cong \PP^1$ a line, which is not possible since $\omega_C \cong \Oo_C(1)$. Thus we have $h^0(\Ii_C(1))=2$ and in particular $C$ is a complete intersection of two hyperplane sections of $X$ since $\Ii_C(1)$ is spanned. But then we would have $\omega_C \cong \Oo_C(2)$ from the minimal free resolution of $\Ii_C$, a contradiction.

Now let us assume that $c_1(\Ee)=2$. If $h^0(\Ee(-1))\ne 0$, then any nonzero section of $\Ee(-1)$ induces an exact sequence
\begin{equation}\label{eqaa1}
0 \to \Oo _X(1) \to \Ee \to \Ii _Z(1)\to 0.
\end{equation}
Since we have $\Ee \cong \Oo_X(1)^{\oplus 2}$ for $Z=\emptyset$, we may assume $Z\ne \emptyset$. In particular $Z$ is a locally complete intersection subscheme of codimension $2$ with $\omega _Z \cong \Oo _Z$ (not necessarily smooth nor connected). Note that $\Ii _Z(1)$ is spanned and so $\deg (Z)\leq 5$. Conversely any $Z$ with $\omega_Z \cong \Oo_Z$ gives a vector bundle $\Ee$ fitting into the sequence (\ref{eqaa1}). Since $h^1(\Oo _X(1)) =0$, $\Ee$ is spanned if
and only if $\Ii _Z(1)$ is spanned. If $\deg (Z)=5$, then $Z$ is the complete intersection of two hyperplane sections of $X$ and we have $h^0(\Ii_Z(1))=2$. Again from the minimal free resolution of $\Ii_Z$, we obtain $\omega_Z \cong \Oo_Z(2)$, a contradiction. If $\deg (Z)\leq 4$, then we have $h^0(\Ii_Z(1))\geq 3$. In particular we have $Z\cong \PP^1$ a line and $h^0(\Ii_Z(1))=3$, which is not possible since $\omega_Z \cong \Oo_Z$.

If $h^0(\Ee(-1))=0$, then we have $\Ee$ as in Example \ref{c1} by the proof of Proposition \ref{X5_2} and by Remark \ref{remc1} we have $\Ee \cong \pi_P^*(\Nn_{\PP^3}(1))$ for a point $P\in \PP^4 \setminus X_5$.
\end{proof}

Let $\Ee$ be a globally generated vector bundle of rank $2$ on a smooth quintic hypersurface $X=X_5\subset \PP^4$ with $c_1(\Ee)=2$ and assume that $H^0(\Ee(-1))=0$, i.e. $h^0(\Ii_C(1))=0$ in the sequence
$$0\to \Oo_X \to \Ee \to \Ii_C(2) \to 0.$$
Let $C=\sqcup C_i$ where each $C_i$ is a smooth and connected curve of degree $d_i$ and genus $g_i=d_i+1$ since $\omega_C\cong \Oo_C(2)$. Since $$\chi(N_{C_i|\PP^4})=(4+1)d_i+(3-4)(g_i-1)=4d_i,$$
so we have $\chi(N_{C|\PP^4})=4d=4\deg (C)$. Note that we have $d\leq 4 \deg (X)=20$ and the equality is realized exactly when $C$ is the complete intersection of two quadric hypersurfaces of $X$, but it is impossible because it would give $\omega_C \cong \Oo_C(4)$ by adjunction.

\begin{proposition}\label{X5_2}
If $\Ee$ is a globally generated vector bundle of rank $2$ on $X_5$ with $c_1(\Ee)=2$ and $h^0(\Ee(-1))=0$, then we have $c_2(\Ee)=10$.
\end{proposition}
\begin{proof}
Let $C_1,\dots ,C_s$ be the connected components of $C$. Set $g_i:= p_a(C_i)$ and $d_i:= \deg ({C})$. Since
$\omega _C\cong \Oo _C(2)$, we have $d_i =g_i-1$ for all $i$. For three general $A_1,A_2,A_3\in |\mathcal {I}_C(2)|$, let $B_i\subset \PP^4$ be the unique quadric hypersurface such that $B_i\cap X =A_i$ as a scheme for $i=1,2,3$. Since $\Ii _C(2)$ is globally generated, we may assume that $A_1\cap A_2 = C\cup T$ where $T$ is a curve and $C\cup T$ contains only finitely many singular points, all of them contained in $C$ by the Bertini theorem. So  $B_1\cap B_2$ is a reduced surface of degree $4$ and the same is true for $B_1\cap B_3$ and $B_2\cap B_3$. Set $ A_1\cap A_3=C \cup T'$ and $A_2\cap A_3=C\cup T''$. For a general $A_3$ the curves $T, T', T''$ have no common component.

\quad{(a)} Assume that $B_1\cap B_2\cap B_3$ contains no surface, except the case giving Example \ref{c1}. In this case $B_1\cap B_2\cap B_3$ is a complete intersection curve of degree $8$ containing $C$. Therefore $\deg ( C) \le 8$ with equality when $C$ is a complete intersection of $3$ quadric hypersurfaces of $\PP^4$. But the equality is impossible, because it would give $\omega _C \cong \Oo _C(1)$ and so we have $d\le 7$. If $C_i$ is a plane curve, then $d_i=5$ and $g_i=6$. No component $C_i$ spans $\PP^3$, because $\pi (x,3) = x-3$ if $3 \le x \le 5$, $\pi (6,3) =4$ and $\pi (7,3) =6$ and hence we would have $d_i<g_i-1$. No component $C_i$ spans $\PP ^4$, because $\pi (x,4) =x-4$ if $4 \le 7$ and hence we would have $d_i<g_i-1$. The only possibility is that $C$ is connected and contain, but it implies $H^0(\Ee(-1)) \ne 0$, contradicting our assumption.

Therefore it is sufficient to prove that $B_1\cap B_2\cap B_3$ contains no surface, unless we are as in Example \ref{c1}.

\quad{(b)} Assume that $B_1\cap B_2\cap B_3$ is the union of the surface $S$ plus something else with lower dimension. Since $B_1\cap B_2$ is a reduced degree $4$ surface, we have $\deg (S) \le 3$ and $S$ is reduced. Since $\Ii _C(2)$ is spanned and $A_1,A_2,A_3$ are general, $T\cap T'$, $T\cap T''$ and $T\cap T''$ are finite and so $(S\cap X)_{red} = C$, using ampleness of the $X$ in $\PP^4$. Since $H^0(\Ee(-1)) =0$, $C$ spans $\PP ^4$ and so we get that $S$ is either

\begin{itemize}
\item the union of two planes with a single intersection point,
\item the union of $3$ planes $M_1\cup M_2\cup M_3$ (not in a $\PP^3$, but the planes may intersect),
\item  the union of an irreducible quadric surface $Q$ and a plane not in the $\PP^3$ spanned by $Q$, or
\item an integral non-degenerate surface of degree $3$ in $\PP^4$.
\end{itemize}

Note that in the first case $C$ is the disjoint union of two quintic plane curves, which is in Example \ref{c1}. Now assume that $C$ is not as in Example \ref{c1}, but that $S$ contains a plane $M$. We get that $C$ has a connected component $M\cap X$, which is a smooth plane quintic.

\quad{(b1)} If $S = M_1\cup M_2\cup M_3$, then we get that $C=C_1\cup C_2 \cup C_3$ is the disjoint union of $3$ plane quintics $C_i=M_i\cap X$. Therefore $\Ii _C(2)$ is spanned if and only if $\Ii _{M_1\cup M_2\cup M_3,\PP ^4}(2)$ is spanned at the points of $X$. If $M_1\cap M_2$ is a line, then $C_1\cap C_2 =X\cap (M_1\cap M_2)\ne \emptyset$. Thus we have $M_i\cap M_j =\emptyset$ or a point for each $i\ne j$. In all cases for a general hyperplane $H\subset \PP^4$ the scheme $H\cap (M_1\cup M_2\cup M_2)$ is a disjoint union of $3$ lines and so it is not scheme-theoretically cut out by quadrics at all points of $X\cap H$. Therefore $\Ii _{M_1\cup M_2\cup M_3,\PP ^4}(2)$ is not globally generated along $X$.

\quad{(b2)} Now assume $S = Q\cup M$ with $Q$ an irreducible quadric and $M$ a plane. Since $C$ is smooth, we get that the scheme $M\cap Q$ is finite. So $S\cap H$ is a disjoint union of a line $L$ and a smooth conic $D$ for a general hyperplane $H\subset \PP^4$. We claim that each quadric surface $Q'\subset H$ containing $H\cap S$ is the union of the plane $\langle D\rangle$ and a plane containing $L$. The claim is true, because $L\cap \langle D\rangle$ is a point not in $\langle D\rangle$, and so $S\cap H$ is not cut out by quadrics at each point of $X\cap \langle D\rangle$, a contradiction.

\quad{(b3)} Now assume that $S$ is an integral and non-degenerate degree 3 surface and then $S$ is either a smooth scroll or a cone over the rational normal curve of $\PP^3$.

\quad{(b3.1)} Assume that $S$ is smooth.  Let $F_1$ be the Hirzebruch surface with minimal self-intersection $-1$, i.e. the blow-up of $\PP^2$ at one point. We have $\Pic (F_1)\cong \mathbb {Z}^{\oplus 2}=\ZZ \langle h,f\rangle$, where $h$ is the section with negative self-intersection and $f$ is a fiber of the ruling of $F_1$. Note that $h^2=-1$, $h\cdot f=1$, $f^2 =0$ and $\omega _{F_1} \cong \Oo _{F_1}(-2h-3f)$. $S$ is obtained by the complete embedding $|h+2f|$. Each linear system $|ah+bf|$ with $b> a$, is very ample and it corresponds to curves of degree $a+b$ with $\Oo _C((a-2)h+(b-3)f)\cong \omega _C$; in the case $b=-a$ we also get a smooth curve with the same degree and genus. Except disjoint unions of lines (which do not arise as components of $C$, because $\omega _C \cong \Oo _C(2)$) all smooth curves in $S$ are connected. Setting $d:= \deg ( C)$ and $\mu \ge 1$ the multiplicity of $X\cap S$ at a general point of $S$, we have $\mu d = 15$ and so we have $(\mu, d)=(1,15)$ or $(3,5)$. We may exclude the latter case, because $\pi (5,4)=1$ and an elliptic or rational curve does not satisfy $\omega _C \cong \Oo _C(2)$. Assuming $d = 15$, i.e. $a+b = 15$, we have
\begin{align*}
30 = 2g-2 &= ((a-2)h+(b-3)f))\cdot (ah+bf)\\
                 &= 2ab-a-2b-a^2\\
                 &=-30a^2+31a-30
\end{align*}
and it has no integral solution.

\quad{(b3.2)} Assume that $S$ is not smooth, i.e. it is a cone over a rational normal curve in $\PP^3$. For the Hirzebruch surface $F_3$, we have $\Pic (F_3)\cong \ZZ\langle h,f \rangle$ with $h^2=-3$, $h\cdot f = 1$, $f^2=0$, $\omega _{F_3} \cong \Oo _{F_3}(-2h-5f)$ and $S$ is obtained by the linear system $|h+3f|$, which contracts $h$ to the vertex $O$ of the cone $S$; to get a smooth divisor $C$ on $S$ we need $C\in |ah+bf|$, $a>0$, $b\ge 3a$, with either $b =3a$ or $b=3a+1$. If $C$ has $s\ge 2$ connected components $C_i\in |a_ih+b_if|$ for $1 \le i\le s$ with $a_1+\cdots +a_s=a$ and
$b_1+\cdots +b_s=b$, then we again have $3a_i \le b_i \le 3a_i+1$ for all $i$. Since $C_i$ is not a single line, so we have $a_i>0$ for all $i$. Note that for $c>0$ and $d>0$, we have
\begin{align*}
(ch+(3c+1)f)\cdot (dh+(3d+1)f)&=3cd+c+d>0\\
(ch+(3c+1)f)\cdot (dh+3df)&=(3c+1)d>0\\
(ch+3cf)\cdot (dh+3df) &=3cd>0.
\end{align*}
Since $C_i\cap C_j =\emptyset$ for all $i\ne j$, we get a contradiction and so $C$ must be connected.
Since $C$ is connected, the scheme $X\cap S$ has the same multiplicity $\mu$. From $\mu d=15$, we get $(\mu ,d) =(1,15)$ or $(3,5)$. As in (b3.1) we exclude the latter case. Now we have $15=d = (ah+bf)\cdot (h+3f) = b$ and it implies $a=5$ since $3a \le b \le 3a+1$. We have $\omega _C\cong \Oo _C(3h+10f)$ and hence
$2g-2 = 50$, i.e. $g=26$, a contradiction.
\end{proof}


\subsection{Case of Higher Rank}$ $

Let $\Ee$ be a globally generated vector bundle of rank $r>3$ on $X$. We know that it fits into an exact sequence
\begin{equation}\label{eqa7}
0 \to \Oo _X^{\oplus(r-3)} \to \Ee \to \Ff \to 0
\end{equation}
where $\Ff$ is a globally generated vector bundle of rank $3$ on $X$ with $c_i(\Ff)=c_i(\Ee )$, $i=1,2,3$.
Conversely, since $h^1(\Oo _X)=0$, if $\Ff$ is a rank $3$ globally generated vector bundle and $\Ee$ is any coherent sheaf fitting into the sequence (\ref{eqn1}), then $\Ee$ is a rank $r$ globally generated vector bundle with
$c_i(\Ee)=c_i(\Ff)$, $i=1,2,3$, and $h^0(\Ee )=h^0(\Ff)+r-3$. This does not give us a complete classification, but only a very rough one unless $h^1(\Ff^\vee )=0$; in this case the sequence (\ref{eqa7}) splits and hence $\Ee \cong \Oo _X^{\oplus (r-3)}\oplus \Ff$.

\begin{proposition}
Let $\Ee$ be a globally generated vector bundle of rank $r\ge 3$ on $X=X_5$ with $c_1(\Ee)=1$ and no trivial factor. Then we have either
\begin{enumerate}
\item $\Ee \cong T\PP^4(-1)_{\vert_X}$, or
\item $\Ee\cong\pi_P^* (T_{\PP^3}(-1))$, where $\pi_P : X_5 \to \PP^3$ is a linear projection from a point $P\in \PP^4 \setminus X_5$. 
\end{enumerate}
In particular there is no such bundles with $r\ge 5$.
\end{proposition}
\begin{proof}
Let $\Ff$ be a globally generated vector bundle of rank $3$ on $X=X_{5}$ with $c_1(\Ee)=1$ and then it fits into the sequence
$$0 \to \Oo_X^{\oplus 2} \to \Ee \to \Ii_C(1) \to 0$$
where $C$ is a smooth curve. If $C=\emptyset$, then we have $\Ff \cong \Oo_X^{\oplus 2} \oplus \Oo_X(1)$ and thus we can assume $C\ne \emptyset$. For each connected component $C'$ of $C$, the bundle $\wedge^2 N_{C'|X} \otimes \Oo_{C'}(-1)\cong \omega_{C'}(-1)$ is globally generated. Since $\Ii_C(1)$ is spanned, so we have $h^0(\Ii_C(1))\geq 2$. If $h^0(\Ii_C(1))\ge 3$, then we have $C \cong \PP^1$ a line and $h^0(\Ii_C(1))=3$. In particular, $\omega_C(-1)\cong \Oo_C(-3)$ is not globally generated. Thus we have $h^0(\Ii_C(1))=2$, i.e. $C$ is the complete intersection of two hyperplane sections of $X$. From the minimal free resolution of $\Ii_C$ we have
\begin{equation}\label{c1=1}
0\to \Oo_X(-1) \to \Oo_X^{\oplus 4} \to \Ff \to 0.
\end{equation}
Indeed, for $4$ general elements of $H^0(\Oo_X(1))$, the corresponding injection $\Oo_X(-1) \to \Oo_X^{\oplus 4}$ gives the cokernel $\Ff$ which is locally free since their common zero $L$ a point is not contained in $X$. In case of $r\ge 4$, the bundle $\Ee$ can fit into the sequence (\ref{eqa7}).
$\Ee$ has $\Oo_X$ as a direct summand if and only if the extension (\ref{eqa7}) is induced by linearly dependent $e_1, \ldots, e_{r-3} \in H^1(\Ff^\vee)$. Since $\Ff$ fits into the sequence (\ref{c1=1})and so $h^1(\Ff^\vee)=1$, so the maximum possible rank of $\Ee$ with no trivial factor is $r=4$ and it fits into the equence
$$0\to \Oo_X(-1) \to \Oo_X^{\oplus 5} \to \Ee \to 0.$$
Thus we have $\Ee \cong T\PP^4(-1)_{\vert_X}$. It also implies that any indecomposable and globally generated vector bundle $\Ff$ of rank $3$ fitting into the sequence (\ref{c1=1} is a quotient of $T\PP^4(-1)_{\vert_X}$ by $\Oo_X$  so $\Ff\cong\pi_P^* (T_{\PP^3}(-1))$.
\end{proof}

Now let us consider the case of $c_1(\Ee)=2$.
\begin{remark}\label{rem2}
By Theorem \ref{thm1}, the only indecomposable vector bundle in the list is fitted into the sequence (\ref{eqn1}) with $(r,c_1)=(2,2)$ and $C=C_1 \cup C_2$ the disjoint union of two plane quintics in planes $U_1, U_2$ respectively. Note that $\dim \Ext^1 (\Ii_C(2), \Oo_X)=h^2(\Ii_C (2))=h^1(\Oo_C(2))=h^0(\Oo_C)=2$ and so any indecomposable bundle $\Ee$ of rank $r\ge 3$ fitting into the sequence (\ref{eqn1}) with the same $C$ must have rank $r=3$ and it is determined by the choice of $U=U_1\cup U_2$. Note that the quotients of $\Ee$ by $\Oo_{X}$ are the bundles in Example \ref{c1}. Since $\Nn_{\PP^3}(1)$ is the quotient of $\Omega_{\PP^3}(2)$ by $\Oo_{\PP^3}$, so we have $\Ee \cong \pi_P^*(\Omega_{\PP^3}(2))$ due to the generalization of \cite[Theorem 1.1 and Remark 1.1.1]{hs}.
\end{remark}

\begin{proposition}\label{proppp}
Let $\Ee$ be a globally generated vector bundle of rank $r\ge 3$ without trivial factors on $X=X_5$ with $c_1(\Ee)=2$. Then we have either
\begin{enumerate}
\item $ 0\to \Oo_X(-2) \to \Oo_X^{\oplus (r+1)} \to \Ee \to 0$ with $3\le r\le 14$,
\item $0 \to \Oo_X(-1)^{\oplus 2} \to \Oo_X^{\oplus(r+2)} \to \Ee \to 0$ with $3\le r\le 8$,
\item $0\to \Oo_X(-1) \to \Oo_X^{\oplus r} \oplus \Oo_X(1) \to \Ee \to 0$ with $3\le r\le 5$,
\item $\Ee \cong \pi_P^* (\Omega_{\PP^3}(2))$, where $\pi_P : X_5 \to \PP^3$ is a linear projection from a point $P\in \PP^4 \setminus X_5$.
\end{enumerate}
Indeed, in $(2)$ and $(3)$, we have
\begin{align*}
(2)\Rightarrow &\left\{
                                           \begin{array}{ll}
                                             \Ee\cong \left[ T\PP^4(-1)_{\vert_X}\right]^{\oplus2}, & \hbox{if $r=8$;}\\                                           
                                             \Ee \text{ is a quotient of it by a trivial bundle}, & \hbox{if $3\le r \le 7$.}
                                           \end{array}
                                         \right.\\
(3)\Rightarrow &\left\{
                                           \begin{array}{ll}
                                             \Ee\cong T\PP^4(-1)_{\vert_X}\oplus\Oo_X(1), & \hbox{if $r=5$;}\\                                           
                                             \Ee \text{ is a quotient of it by a trivial bundle}, & \hbox{if $3\le r \le 4$.}
                                           \end{array}
                                         \right.                    
\end{align*}
\end{proposition}

\begin{proof}
Since $c_1(\Ee)=2$, so $\Ee$ fits into the sequence
$$0\to \Oo_X^{\oplus (r-1)} \to \Ee \to \Ii_C(2) \to 0. $$
If $C=\emptyset$, then we have $\Ee \cong \Oo_X^{\oplus (r-1)} \oplus \Oo_X(2)$ and so we can assume $C\ne \emptyset$. Then $C$ is a smooth curve of degree $d$ at most $4\deg (X)=20$ such that $\wedge^2 N_{C'|X} \otimes \Oo_{C'}(-2)\cong \omega_{C'}(-2)$ is globally generated for each component $C'$ of $C$. If $d=20$, i.e. $C$ is the complete intersection of two hypersurfaces of degree $2$ in $X$ and so $h^0(\Ii_C(2))=2$, then from the minimal free resolution of $\Ii_C$ we have
$$0\to \Oo_X(-2) \to \Oo_X^{\oplus (r+1)} \to \Ee \to 0.$$
Conversely any $r+1$ general section in $H^0(\Oo_X(2))$ whose common zero is empty, defines a locally free sheaf $\Ee$ of rank $r$. If $r=3$, we have $h^1(\Ee^\vee)=11$ and so the maximum possible rank $r$ is $14$.

Now assume $h^0(\Ii_C(2))\ge 3$ and as in the proof of Proposition \ref{X5_2}, let us choose the unique quadric hypersurface $B_i\subset \PP^4$ such that $B_i \cap X = A_i$ for three general $A_1, A_2, A_3 \in |\Ii_C(2)|$. Let $C = \sqcup _{i=1}^{s} C_i$, be a smooth curve with $s$ connected components. Set $d_i:= \deg (C_i)$ and $g_i:= p_a(C_i)$.
Since $\omega _C(-2)$ is spanned, we have $d_i\leq g_i-1$ for all $i$.

\quad{(a)} If $B_1 \cap B_2 \cap B_3$ contains no surface, then we have $\deg (C)\le 8$ with equality when $C$ is a complete intersection of $3$ quadric hypersurfaces of $\PP^4$. But the equality is impossible, because it would give $\omega_C(-2) \cong \Oo_C(-1)$ which is not globally generated. Thus we have $d \leq 7$. Let $r_i$ be
the dimension of the linear span of $C_i$.  For all integers $x\ge y\ge 2$ let $\pi (x,y)$ denote the maximal genus of a non-degenerate curve of degree $x$ in $\PP^y$ (\cite[Chapter III]{he}). We have $r_i\ne 4$, because for $4\le x \le 7$ a non-degenerate curve of $\PP^4$ with degree $x$
has at most genus $x-4$. We have $r_i\ne 3$, because $\pi (x,3) =x-3$ if $3\le x \le 5$, $\pi (6,3) =4$ and $\pi (7,3) =6$. If $r_i=2$, we have $d_i\geq 4$. Hence $s=1$ and $C$ is a plane curve
of degree $d\geq 4$. We have $d<6$, because $X_5$ contains no plane. Assume for the moment $d=4$ and call $\Pi \subset \PP^4$ the plane spanned
by $C$. We have $\Pi \cap X = C\cup T$ with $T$ a line. Since $\deg (T\cap C)=4$, $\Ii _C(2)$ is not globally generated, a contradiction.
Hence $C$ is a connected plane curve of degree $d\ge 5$. Since there is no plane contained in $X$, so $C$ is contained in a complete intersection of two hyperplane sections. From the minimal free resolution of $\Ii_C$ we have
$$0\to \Oo_X \to \Oo_X^{\oplus (r-1)} \oplus \Oo_X(1)^{\oplus 2} \to \Ee \to 0.$$
Thus we have $\Ee \cong \Oo_X^{\oplus (r-2)} \oplus \Oo_X(1)^{\oplus 2}$ which is decomposable.

\quad{(b)} Assume that $B_1 \cap B_2 \cap B_3$ is the union of the surface $S$ plus something else with lower dimension. Then $S$ is reduced with $\deg (S)\le 3$.

\quad{(b1)} If $H^0(\Ee(-1))=0$, then $C$ spans $\PP^4$ and so $S$ can be one of the $4$ types in (b) of the proof of Proposition \ref{X5_2}. Using the exact same argument, $S$ cannot be the union of $3$ planes or the union of an irreducible quadric surface $Q$ and a plane not in the $\PP^3$ spanned by $Q$. If $S$ is an integral and nondegenerate surface of degree $3$, then $S$ is either a smooth scroll or a cone over the rational normal curve of $\PP^3$. In the former case, using the notations in (b3.1) of Proposition \ref{X5_2}, we have $d=a+b=15$ and from the global generatedness of $\omega_C(-2)$ we have $a\ge 4$ and $30a^2-31a+60 \leq 0$, which is not possible. When $S$ is a cone over a rational normal curve in $\PP^3$, then using again the notations in (b3.2) of Proposition \ref{X5_2}, we get $(a,b)=(5,15)$, i.e. $C\in |5h+15f|$ connected. From the minimal free resolution of $\Ii_C$
$$0\to \Oo_X(-3)^{\oplus 2} \to \Oo_X(-2)^{\oplus 3} \to \Ii_C \to 0,$$
we obtain the sequence (2) in the assertion.

Assume now that $S$ is the union $U=U_1\cup U_2$ of two planes with a single intersection point. Let $C_i=U_i \cap X$ be a plane quintic and set $C=C_1 \cup C_2$. Then by Example \ref{c1}, the sheaf $\Ii_C(2)$ is spanned and a bundle fitting into the sequence (\ref{eqn1}) is also a globally generated vector bundle of rank $r$. By Remark \ref{rem2} we have $r=3$ and $\Ee \cong \pi_P^* (\Omega_{\PP^3}(2))$.

\quad{(b2)} Assume that $H^0(\Ee(-1))\ne0$, i.e. $C$ is degenerate. Thus $C$ is contained in the complete intersection of a hyperplane section $W:=X \cap H$ and a quadric hypersurface of $X$, and in particular we have $d\leq 10$. If $d=10$, then from the minimal free resolution of $\Ii_C$ we have the sequence (3) in the assertion. Assume $d\leq 9$ and then $C$ is connected. In particular we have $\dim (\langle C \rangle )=3$ and $8 \le d \le 9$. If $d=8$, we have $g=9$ since $\pi(8,3)=9$. Since $\Ii_C(2)$ is spanned, so for the base locus $S'$ of $\Ii_{C, H}(2)$, we have $C=S' \cap W$. We can assume that $S'$ is a surface since we would have $\deg (C)\leq 4$ if $S'$ is a curve. If $S'$ is a quadric surface, i.e. $h^0(W, \Ii_{C,W}(2))=1$, then $C$ is the complete intersection of a quadric and a hyperplane section in $X$, excluding the case. If $S'$ contains a plane $M$, then $C$ has a plane quintic $M\cap X$ as a component, a contradiction.
\end{proof}

\begin{remark}
The only decomposable bundles in Proposition \ref{proppp} arise in cases $(2)$ and $(3)$:
\begin{align*}
&\text{case }(2)~:~\left[T\PP^4(-1)_{\vert_X}\right]^{\oplus 2} ,~T\PP^4(-1)_{\vert_X}\oplus \pi_P^* (T_{\PP^3}(-1)),~\left[ \pi_P^* (T_{\PP^3}(-1))\right]^{\oplus 2}\\
&\text{case }(3)~:~T\PP^4(-1)_{\vert_X}\oplus \Oo_X(1),~\pi_P^* (T_{\PP^3}(-1))\oplus \Oo_X(1)
\end{align*}
\end{remark}

\section{CICY of codimension $2$}

In this section we consider the CICY of codimension $2$, i.e. $X=X_{2,4}$ or $X_{3,3}$. Let us start with the case of bundles with $c_1(\Ee)=1$.

\begin{proposition}\label{prop+++}
Let $\Ee$ be a globally generated vector bundle of rank $2$ on $X=X_{2,4}$ or $X_{3,3}$ with $c_1(\Ee)=1$. Then we have
\begin{enumerate}
\item $\Ee \cong \Oo_X\oplus \Oo_X(1)$ or
\item $X=X_{2,4}$ and $\Ee$ fits into the exact sequence
$$0\to \Oo_X \to \Ee \to \Ii_C(1) \to 0$$
where $C$ is a plane quartic.
\end{enumerate}
\end{proposition}
\begin{proof}
We can assume that $h^0(\Ee(-1))=0$ and so $\Ee$ fits into the sequence
\begin{equation}\label{eqnc1=1}
0\to \Oo_X \to \Ee \to \Ii_C(1) \to 0,
\end{equation}
where $C$ is a smooth curve on $C$ with $\omega_C\cong \Oo_C(1)$ and $c_2(\Ee)=\deg (C)$. In particular, we have $p_a(C)=\deg (C)/2+1\ge 3$. Since $C$ is contained in a complete intersection of two hyperplane sections of $X$ and so $\deg (C)\in \{4,6,8\}$.

First let us assume that $X=X_{2,4}$. If $\deg (C)=8$, then $C$ is a complete intersection of two hyperplane sections of $X$ and so we have $\omega_C \cong \Oo_C(2)$ from the minimal free resolution of $\Ii_C$, a contradiction. If $\deg (C)=6$, $C$ is a canonically embedded curve in $\PP^3$. In particular it cannot be contained in a plane and so $h^0(\Ii_C(1))=2$. But it is not possible since $\Ii_C(1)$ is spanned and $C$ is not a complete intersection. If $\deg (C)=4$, then $C$ is a plane quartic. Instead of showing the existence of such a curve, let us consider the intersection $C'$ of a plane $L\subset U_2$ with $U_4$. We have $h^0(\Ii_{C'}(1))=3$ and $C'$ is the scheme-theoretic intersection inside $U_4$. Thus $\Ii_{C'}(1)$ is spanned and it implies the existence of globally generated $\Ee$ in the sequence (\ref{eqnc1=1}). Moreover $\Ee$ is locally free, because $C'$ is a locally complete intersection and $\omega_{C'}\cong \Oo_{C'}(1)$.

In the case of $X=X_{3,3}$, we can follow verbatim the previous argument on $X_{2,4}$. Indeed the case of $\deg (C)=4$ does not occur, since a smooth $X_{3,3}$ contains no plane.
\end{proof}

\begin{remark}
A priori the proof above implies the existence of a smooth plane quartic in every $X_{2,4}\subset \PP^5$. In fact, we can show the existence directly for general $X_{2,4}$; for a plane $L\subset U_2$, we have $L\nsubseteq U_4$, because $X$ is smooth. Now if $X_{2,4}$ contains only finitely many lines, then we can take $L$ so that $L\cap U_4$ contains no lines. Thus we have a plane $L$ in $U_2$ such that $L \cap U_4$ is smooth. Dr. Kiryong Cheong pointed out that the bundle in (2) of Proposition \ref{prop+++} is the bundle $\mathrm{B}$ in \cite[Theorem 3.55]{thomas}.

\end{remark}

Now let us deal with the case $c_1(\Ee)=2$. Again we can assume that $H^0(\Ee(-2))=0$. If $H^0(\Ee(-1))\ne 0$, then any nonzero section of $\Ee(-1)$ induces an exact sequence
$$0 \to \Oo_X(1) \to \Ee \to \Ii_Z(1) \to 0. $$
If $Z=\emptyset$, then we have $\Ee \cong \Oo_X(1)^{\oplus 2}$ and so we may assume $Z\ne \emptyset$. In particular, $Z$ is a locally complete intersection subscheme of codimension $2$ with $\omega_Z \cong \Oo_Z$.

\begin{remark}
If $\Ee \not\cong \Oo _X(1)^{\oplus 2}$, then we have $h^0(\Ee(-1)) =1$, i.e. the bundle $\Ee$ is determined by the choice of $Z$ uniquely.
\end{remark}

\begin{lemma}\label{X_{3,3}1}
There exists a plane cubic curve on every $X_{3,3}$.
\end{lemma}
\begin{proof}
Let $\PP^{55}$ be the projective space of all cubic hypersurfaces of $\PP^5$ and $G(2,5)$ be the Grassmannian of all planes in $\PP^5$. For any plane $V\in G(2,5)$, the set of all cubics in $\PP^5$ containing $V$ is a projective space of dimension $45$. Since $\dim G(2,5) =9$ and the flat limit of a family of planes in a plane, we get that the set $\Sigma \subseteq \PP^{55}$ of all cubic hypersurfaces of $\PP^5$ containing at least one plane is a nonempty hypersurface of $\PP^{55}$. It implies that
any pencil of cubic hypersurfaces meets $\Sigma$ and so there is $T\in \PP (H^0(\PP^5,\mathcal {I}_{X_{3,3}}(3)))$ with $T$ containing a plane $V$. Take any cubic hypersurface $T'\subset \PP^5$ with $X_{3,3}=T\cap T'$. Then we get that $X_{3,3}$ contains a plane cubic $Z:= T'\cap V$.
\end{proof}

\begin{proposition}\label{prop4.5}
Let $\Ee$ be a globally generated vector bundle of rank $2$ without trivial factors on $X=X_{2,4}$ or $X_{3,3}$ with $c_1(\Ee)=2$ and $h^0(\Ee(-1))>0$. Then it fits into the sequence
$$0\to \Oo_X(1) \to \Ee \to \Ii_Z(1) \to 0$$
where $Z$ is a plane cubic curve.
\end{proposition}
\begin{proof}
On $X=X_{2,4}=U_2 \cap U_4$, we have $h^0(\Ii _Z(1)) \le 3$ and so $h^0(\Ii _Z(1)) =3$. It implies that $Z$ should be a plane cubic; Let us start with a line $D\subset X$. It exists on all $X$ by \cite{dm}. Then we have a $2$-dimensional family of planes $\Pi \subset U_2$ and so we have $\Pi \cap U_4=D \cup Z$ with $Z$ a plane cubic.

Similarly on $X=X_{3,3}$ we have $h^0(\Ii _Z(1)) =3$ and so $Z$ is a plane cubic and such curves always exist on every $X_{3,3}$ by Lemma \ref{X_{3,3}1}.
\end{proof}

Now we have the following examples of globally generated vector bundles on $X$ with $c_1(\Ee)=2$ and $H^0(\Ee(-1))=0$.

\begin{example}\label{b1}
Let us generalize Example \ref{c1} to higher dimensional case. Let us fix two linear subspaces $U_1,U_2\subset \PP^n$ of codimension $2$ such that $U:= U_1\cup U_2$ spans $\PP^n$ (or equivalently $\dim (U_1\cap U_2) = n-4$) and $U$ is transversal to $X$, i.e. each $U_i$ is transversal to $X$
and $U_1\cap U_2\cap X=\emptyset$. Setting $C:= U\cap X$ and $C_i:= X\cap U_i$, each $C_i$ is a smooth complete intersection and $\omega _{C_i} \cong \Oo _{C_i}(2)$ by the adjunction formula.
Therefore to get the globally generated bundle $\Ee$ associated to $C$ it is sufficient to prove that $\Ii _{U,\PP^n}(2)$ is spanned outside $U_1\cap U_2$. Every quadric hypersurface containing $U$ is a cone with $U_1\cap U_2$ contained in its vertex. Therefore it is sufficient to prove that $\Ii _{U\cap M,M}(2)$ is spanned for any $3$-dimensional linear subspace $M\subset \PP^n$ such that $M\cap U_1\cap U_2=\emptyset$. It is true because $U\cap M$ is the union of two disjoint lines.
\end{example}

\begin{example}\label{b2}
Let us consider a complete intersection $C$ of $4$ quadrics in $\PP^5$. We have $\omega _C \cong \Oo _C(2)$ and $\Ii _C(2)$ is spanned in any $X$ containing $C$. The cohomology groups of the sheaves $\Ii_{C, \PP^5}(t)$ give $h^0(\Oo_C)=0$ and in particular $C$ is connected. Indeed such curves exist on some $X_{2,4}$ and $X_{3,3}$ by the following Lemmas \ref{lemb2_1}, \ref{lemb2_2}.
\end{example}

\begin{lemma}\label{lemb2_1}
There exist some $X_{2,4}$ containing $C$ in Example \ref{b2}.
\end{lemma}
\begin{proof}
For a smooth complete intersection surface $S$ of $3$ quadric hypersurfaces in $\PP^5$, it is sufficient to prove that $S$ is contained in a quartic hypersurface with finitely many singular points. Indeed $S$ is contained in a smooth quartic; Fix a general $W\in |\Ii _{S,\PP^5}(4)|$. Since $\Ii _{S,\PP^5}(2)$ is generated by quadrics, $S$ is cut out by quartic hypersurfaces and so by the Bertini theorem, $W$ is smooth outside $S$. For any point $P\in S$, define a set
$$S_P :=\{ A\in |\Ii_{S,\PP^5}(4)|~|~ A \text{ is singular at } P\}.$$
For any $O\in \PP^5$, let $O^2$ be the closed subscheme of $\PP^5$ with $\Ii _{O,\PP^5}^2$ as its ideal sheaf. Then $O^2$ is a $0$-dimensional scheme with  $O^2_{red} = \{O\}$ and $\deg (O^2) =6$. If $T\subset \PP^5$ is any hypersurface containing $O$, then $T$ is singular at $O$ if and only if $O^2\subset T$. For any $P\in S$, let $T_PS \subset \PP^5$ be the tangent plane to $S$ at $P$. Since $\dim (S) =2$, to prove that $W$ is smooth at each point of $S$ (resp. smooth outside finitely many points of $S$) it is sufficient to prove that
\begin{align*}
&h^0(\Ii _{P^2\cup S,\PP^5}(4)) \le h^0(\Ii _{S,\PP^5}(4)) -3\\
\text{(resp. } &h^0(\Ii _{P^2\cup S,\PP^5}(4)) \le h^0(\Ii _{S,\PP^5}(4)) -2~)
\end{align*}
for all $P\in S$. For a fixed point $P\in S$, there is a smooth $Q'\in  |\Ii _{S,\PP^5}(2)|$ with $Q'$ smooth at $P$ since $S$ is a complete intersection. Let $M_3\subset T_PQ'$ be a $3$-dimensional linear subspace such that $M_3\supset T_PS$.
To prove that $h^0(\Ii _{P^2\cup S,\PP^5}(4)) \le h^0(\Ii _{S,\PP^5}(4)) -3$, it is sufficient to find $W_1,W_2,W_3\in |\Ii _{S,\PP^5}(4)|$ such that
$$W_1\supset T_PS~,~W_1\nsupseteq M_3~,~W_2\supset M_3~,~W_2\nsupseteq M_3~,~W_3\supset T_PQ'$$
and $W_3$ smooth at $P$. Since $S$ is the scheme-theoretic intersection of quadrics, there is such a quadric $Q\supset S$ with $M_1\nsubseteq Q$. Choose hyperplanes $H_1,H_2, H_3$ in $\PP^5$ with $P\notin H_1$ and $P\notin H_2\ne H_3 \supset M_3$. Take $Q\cup H_1\cup H_2$ as $W_1$, $Q\cup H_2\cup H_3$ as $W_2$ and $Q'\cup 2H_2$ as $W_3$.
\end{proof}

\begin{lemma}\label{lemb2_2}
There exist some $X_{3,3}$ containing $C$ in Example \ref{b2}
\end{lemma}
\begin{proof}
Let $\Lambda$ be the family of complete intersection curves of $4$ quadric hypersurfaces in $\PP^5$ and $\Delta$ be the space of cubic hypersurfaces in $\PP^5$. Then we can define an incidence variety $\Gamma$ to be the set of all $(T, Y)$ with $T\in \Lambda$ and $Y\in \Delta$ such that $T\subset Y$ with two natural projections:
$$\xymatrix{ & \Gamma' \subset \Gamma \ar[ld]_{p_1} \ar[rd]^{p_2}\\
\Lambda' \subset \Lambda && \Delta.}$$

For each $T\in \Lambda$, even reducible and/or with multiple components, we have $h^0(\PP^5, \Ii_T(3))=24$ and so the projection $p_1: \Gamma \to \Lambda$ is induced by a vector bundle $\Ff$ of rank $24$, i.e. we have $\Gamma \cong \PP (\Ff)$ a projective bundle whose fibre is $\PP^{23}$. In particular $\Gamma$ is irreducible and with dimension $\dim (\Lambda)+23=91$.

Let $\Lambda'$ be the subscheme of $\Lambda$ consisting of smooth curves and $\Gamma':=p_1^{-1}(\Lambda')$. Fix a general cubic hypersurface $Y\subset \PP^5$. By \cite[Proposition 5.6]{ccg}, $Y$ contains some $T\in \Lambda'$ and so the projection onto the second factor induces a dominant map $\pi : \Gamma '\to \Delta$. Thus a general $Y\in \Delta$ is an image of a general element of $\Gamma '$, i.e. it contains some $C\in \Lambda$ and $\dim (\pi ^{-1}(Y)) = 91-55 = 36$. In other words the family of such curves contained in $Y$ is $36$-dimensional. Note that $Y$ has at most finitely many singular points, none of them contained in $C$. For a general cubic hypersurface $Y' \subset \PP^5$ containing $C$, let $W:=Y \cap Y'$. Then we will show that $W$ is smooth and take $W$ as $X_{3,3}$.

For any $P\in C$ let $E( P)$ be the set of all cubic hypersurfaces containing $C\cup T_PY$. This is a closed subset of the projective space $|\Ii _{C,\PP^5}(3)|$. By the Bertini theorem $W$ is smooth outside $C$. Since $\dim ( C)=1$, it is sufficient to prove that each $E( P)$ has codimension at least $2$ in $|\Ii _{C,\PP^5}(3)|$. Fix a hyperplane $H\subset \PP^5$ such that $P\notin H$. Since $C$ is smooth and it is a complete intersection, each quadric hypersurface containing $C$ is smooth at each point of $C$. Hence the $3$-dimensional linear subspace of cubics $\{Q\cup H~|~Q\in |\Ii _{C,\PP^5}(3)|\}$ is disjoint from each $E( P)$ and so each $E( P)$ has codimension at least $4$ in $|\Ii _{C,\PP^5}(3)|$.
\end{proof}

\begin{example}\label{b3}
Let $S\subset \PP^5$ be a surfaces of degree $5$ with $\omega_S \cong \Oo_S(-1)$, possibly with only finitely many singular points, i.e. a weak del pezzo surface of degree $5$. Then the sheaf $\Ii _{S,\PP^5}(2)$ is spanned. If we take a cubic hypersurface $U_3\subset \PP^5$ such that $C:= S\cap U_3$ is smooth, we have $\omega _C\cong \Oo _C(2)$ and for each $X =U_3\cap U'_3$ the sheaf $\Ii _C(2)$ is spanned. Note that $C$ is connected by the Kodaira's vanishing theorem, since $C$ is an ample divisor of $S$.
\end{example}

\begin{example}\label{inc}
Let $Y\subset \PP^5$ be a complete intersection of $3$ quadrics $Q_1, Q_2, Q_3$ and one cubic $U$. Then $Y$ is a curve of degree $24$ with $\omega_Y \cong \Oo_Y(3)$. If $Y=C\cup D$ with $\deg (C)=d$ and $D$ smooth outside $C\cap D$, then we have ${\omega_Y}_{\vert_C} \cong \omega_C (C \cap D)$ and so $\deg (C \cap D)=d$. If $C$ is cut out scheme-theoretically by $U$ and another cubic $U'$ inside $S:=Q_1\cap Q_2 \cap Q_3$, then we have $\deg (U'\cap D)=3(24-d)$ and so $d=18$. In this way we may obtain $C$ in $X_{3,3}:=U \cap U'$ such that $\Ii_C(2)$ is spanned with $\omega_C \cong \Oo_C(2)$. In other words it gives an example of connected curves with degree $d =g-1=18$, i.e. with $H^1(\Ee ^\vee )=0$, on some $X_{3,3}$.
\end{example}

From a general section in $H^0(\Ee)$, we obtain the following exact sequence
$$0\to \Oo_X \to \Ee \to \Ii_C (2) \to 0,$$
where $C$ is a smooth curve with $\omega_C \cong \Oo_C(2)$. Let $C = \sqcup _{i=1}^{s} C_i$, $d_i =\deg (C_i) =g_i-1$ with $g_i = p_a(C_i)$, $d=g-1$, be a smooth curve with $s$ connected component. Note that $d:=\deg (C) \leq 4 \deg (X)$ since $C$ is contained in the complete intersection of two quadric hypersurface of $X$. If $d=4 \deg (X)$ , then we have $\omega_C \cong \Oo_C(4)$ from the minimal free resolution of $\Ii_C$, a contradiction. So we have $d<4 \deg (X)$.

\begin{remark}
If $C$ has a plane curve of degree at least $5$ as a component, say $C_1$, then $X$ contains the plane $\langle C_1 \rangle$ since each equation of $X$ has degree at most $3$, a contradiction. Thus each $C_i$ is not contained in a plane, i.e. $\dim \langle C_i \rangle \ge 3$ and so we have $d_i \geq 8$.
\end{remark}

Let us define two fundamental subschemes of $\PP^n$ associated to $C$;
\begin{align*}
\Psi &:=\text{ the scheme-theoretic base locus of }H^0(\PP^n,\mathcal {I}_{C,\PP ^n}(2)),\\
\Phi &:=\text{ the union of the irreducible components of $\Psi _{red}$ containing $C$.}
\end{align*}

Since the map $\rho :H^0(\PP^n,\mathcal {I}_{C,\PP ^n}(2)) \to H^0(\Ii _C(2))$ is surjective, we have $\Psi \cap X=C$ as schemes. Thus $C$ is scheme-theoretically cut out inside $\PP^n$ by quadric equations and the equations of $X$. Now we have $\deg (\Phi ) \le 2^{n-\dim (\Phi)}$ and the equality holds if and only if $\Phi =\Psi$ is equidimensional and the complete intersection of $n -\dim (\Phi)$ quadric hypersurfaces. For each component $C_i$ of $C$, let $S_i$ be a fixed reduced and irreducible component $S_i \subseteq \Psi$ containing $C_i$ (a priori it could be $S_i=S_j$ even if $i \ne j$).


\section{CICY of degree $8$}

 Let us assume that $X=X_{2,4}$ is a CICY of degree $8$. If the number $s$ of component in $C$ is at least $2$ and $d\leq 16=2\deg (X)$, then we get $s=2$ and $d_i=8$ for all $i$ since $d_i\ge 8$ for all $i$. This is the case in Example \ref{b1} by the classification of space curves with extremal genus. Therefore we may assume that either $C$ is connected, i.e. $s=1$, or $s\geq 2$ and $d\geq 17$.

Our main goal in this section is to prove the following assertion.
\begin{proposition}
If $\Ee$ is a globally generated vector bundle of rank $2$ without trivial factors on $X_{2,4}$ with $c_1(\Ee)=2$ and $h^0(\Ee(-1))=0$, then $\Ee$ is one of the bundles in Example \ref{b1} and \ref{b2}, except when $c_2(\Ee)=16$.
\end{proposition}

\subsection{Case of $s=1$} $ $

Let us assume first that $C$ is connected. Since $H^0(\Ii _C(1)) =0$, we get $d \ge 14$. We know that $C$ is the scheme-theoretic intersection of $\Phi$ and a quartic hypersurface and so $\dim (\Phi )\le 2$.

\begin{lemma}\label{lemma5.2}
If $\dim (\Phi)=1$, then $C$ is the complete intersection of $4$ quadrics in $\PP^5$.
\end{lemma}
\begin{proof}
Assume that $\Phi$ is a curve; in our set-up we get $\Phi =C$. Since $\Phi$ is cut out by quadrics in $\PP^5$, we have $\deg (\Phi ) \le 2^4$. Now assume $d<16$ and let $E$ be the intersection of $4$ general quadric hypersurfaces containing $C$. Since $C$ is the scheme-theoretic intersection of quadric hypersurfaces, $\dim (E)$ is a curve, $E$ is the complete intersection of $4$ hypersurfaces of degree 2 and $E = C\cup F$ with $F$ a scheme of dimension $1$ with $\deg (F) =16-d$, i.e. $F$ is a line if $d=15$, while $F$ is either a double structure on a line or the union of two disjoint lines or a plane conic if $d=14$. Since $E$ is a complete intersection, the adjunction formula gives $\omega _E \cong \Oo _E(2)$ and so ${\omega _E}_{\vert_C}\cong \Oo _C(2)$. Since $E$ is a complete intersection, the cohomology groups of the sheaves $\Ii _{E,\PP^5}(t)$ give $h^0(\Oo _E)=1$. Therefore $F\cap C \ne \emptyset$ and so $\omega _C \ne {\omega _E}_{\vert_C} \cong \Oo _C(2)$, a contradiction. If $d=16$, we get that $C$ is the complete intersection of $4$ quadrics hypersurfaces, i.e. $C$ is as in Example \ref{b2}.
\end{proof}

Now assume $\dim (\Phi)=2$, i.e. $\Phi$ is a surface. Since $\Psi$ is cut out by quadrics, we have $\deg (\Psi )\le 8$. Conversely for any irreducible surface $S\subset \PP^5$ cut out scheme-theoretically by quadrics, the sheaf $\Ii _{X\cap S}(2)$ is globally generated. To get a bundle on $X_{2,4}$ we need the scheme-theoretic intersection $C$ of $S$ with a hypersurface of degree $4$
and in particular we get $d\equiv 0 \pmod{4}$. The case $d=32$ is excluded, because it would give $\Phi$ a complete intersection of quadrics and so $\omega _C \cong \Oo _C(4)$. Since $\Phi \cap X =C$ as schemes, $S$ must be smooth at each point of $X$ and it is irreducible. Since $S$ is irreducible and it spans $\PP^5$, then $\deg (S)\ge 4$ and hence $d\ge 16$. In the case $d=16$ we also get that $S$ is the complete intersection of $3$ quadrics and so $\omega _C \cong \Oo _C(4)$, a contradiction. Thus we have $5 \leq \deg (S) \leq 7$.

\begin{lemma}\label{lemma5.3}
We have $\deg (S)\ne 5$.
\end{lemma}
\begin{proof}
Assume $\deg (S) = 5$ and then we have $\deg (S) =\mathrm{codim} (S) +2$. So $S$ is an almost minimal varieties. Note that if $S$ is smooth, then $h^1(\omega _S(1)) =0$.

\quad (a) Assume $h^0(\Oo _S(1)) =6$, i.e. that the $\Delta$-genus of $S$ is $2+5-6=1$ (in the sense of \cite{f1} and \cite{f2}). Let $H\subset \PP^5$ be a general hyperplane. Since $S$ has only finitely many singularities, $H\cap S$ is a smooth curve of degree $5$ spanning $H= \PP^5$. Thus $H\cap S$ has genus $0$ or $1$. If $H\cap S$ has genus $1$, then $(S,\Oo _S(1))$ has sectional genus $1$. Since $\Oo _S(1)$ is very ample, we have $\omega_S \cong \Oo_S(-1)$ by \cite[Corollary 6.5 at page 46]{f2} and so $\omega _C\cong \Oo _C(3)$, a contradiction. Now assume that $H\cap S$ has genus $0$. Since $S$ has only finitely many singular points with none of them on $C$, so the torsion free sheaf $\omega _S$ is defined and $\omega _C \cong \omega _S(4)_{\vert_C}$. Since $C$ is an ample divisor on $S$, to get a contradiction, it is sufficient to prove that $\omega _S(2)$ has a section whose zero-locus $T$ is nonempty. For any integer $t\in \ZZ$ we have an exact sequence
\begin{equation}\label{eqb1}
0 \to \omega _S(t-1) \to \omega _S(t) \to \omega _{S\cap H}(t-1) \to 0.
\end{equation}
Since $S\cap H$ is a smooth rational curve of degree $5$, so we have $h^0(\omega _{S\cap H}) =0$, $h^1(\omega _{S\cap H}(i)) =0$ for all $i>0$ and $h^1(\omega _{S\cap H}) =1$. Therefore $h^0(\omega _S(1)) =0$ and the restriction map $H^0(\omega _S(2)) \to H^0(\omega _S(2)_{\vert_{S\cap H}})$ is injective. Since each section of $\omega _{S\cap H}(1)$ has at least one zero, it is sufficient to prove $H^0(\omega _S(2)) \ne 0$. By \cite[9.2 at page 78]{f2}, $S$ is not
normal and its normalization $S'$ has $\Delta$-genus $0$. Thus for the pull-back $\Rr_{S'}$ of $\Oo _S(1)$ to $S'$, we have $h^0(\Rr_{S'}) =7$. If $\Rr_{S'}$ is very ample with image not a cone, we can get a contradiction as in (b) ($S$ is smooth at the points of $C$ and so we may use that formula to compute $\omega _C =\omega _S(4)_{\vert_C}$). Since the map $S'\to S$ is finite, $\Rr_{S'}$ is spanned and it induces a finite birational morphism; hence if the image of $S'$ is a cone, then $S$ is a cone over a non-linearly normal smooth rational curve of $\PP ^4$. To compute $\omega _S(2)$ we may use the associated cone $S''\subset \PP^6$, because $C$ does not contain the vertex of $S$.
$S''$ is the image of $S' = F_5$ by the complete linear system $|\Oo _{F_5}(h+5f)|$ and we have
$\omega _{F_5} \cong \Oo _{F_5}(-2h-7f)$. We get again $\omega _S(2) \cong \Oo_{F_3}(3f)$, a contradiction.

\quad (b) Assume $h^0(\Oo _S(1)) \ge 7$. We are in the case of $\Delta$-genus $0$ and necessarily we have $h^0(\Oo_S(1)) =7$. $S$ is now an isomorphic projection from a minimal degree surface $S'$ of $\PP^6$; since this is linear projection, $S'$ is not a cone and hence it is a smooth surface scroll either isomorphic to $F_1$ embedded by the complete linear system $|h+3f|$ or to $F_3$
embedded by the complete linear system $|h+4f|$. In the first case we have $\omega _S \cong \Oo _S(-2h-3f)$ and so $h^0(\omega _S(2)) = 4$, excluding this case for $C$. In the second case we have $\omega _S\cong \Oo _S(-2h-5f)$ and so $h^0(\omega _S(2)) =4$, again excluding this case.
\end{proof}

\begin{lemma}\label{lemma5.4}
We have $\deg (S)\ne 6$.
\end{lemma}
\begin{proof}
Assume $\deg (S) =6$. Again, to get a contradiction we need to prove that $H^0(\omega _S(2))$ has a non-zero section with non-zero locus. Let $H\subset \PP^5$ be a general hyperplane. Since $S$ has finitely many singular points, $D$ is a smooth, irreducible and nondegenerate curve of $H =\PP^4$. Let $g(D)$ be its genus. Since $\pi (6,4) =2$, we have $g(D)\le 2$. Since $S$ is cut out scheme-theoretically by quadrics, there is no line $L\subset H$ with $\deg (L\cap D)\ge 3$ ($S$ may contain line, but if $L$ is a line with $\deg (L\cap S)\ge 3$, then $L\subset S$). The Berzolari formula in \cite{ho} for the number of trisecant lines to $D$ gives $g(D)=2$ and so $D$ is projectively normal. Since $h^0(\Oo _D(1)) =5$, then $h^0(\Oo _S(1)) =6$. Since $h^0(H,\Ii _D(2)) =4$, we get $h^0(\Ii _S(2)) =4$.

If $S$ is smooth, then $h^1(\omega _S(1))=0$ by the Kodaira's vanishing theorem and so $h^0(\omega _S(2)) =h^0(\omega _D(1)) \ge 2$, excluding this case.

Assume that $S$ has at least one singular point and let $\pi : T \to S$ be a desingularization map. By the Grauert-Riemenschneider vanishing theorem in \cite{ev}, the sheaf $\pi _\ast (\omega _T)$ on $S$ satisfies the usual Kodaira's vanishing statements, i.e. $h^i(\pi _\ast (\omega _T)\otimes \Rr) =0$ for all $i>0$ and all ample line bundles $\Rr$ on $S$. In particular we have $h^1(\pi _\ast (\omega _T)(1)) =0$. Since $S$ is smooth at all points of $C$, so $\omega _S$ and $\pi _\ast (\omega _T)$ are the same in a neighborhood of $C$. Thus we have $\pi _\ast (\omega _T)(4)_{\vert_C} \cong \omega _C$. Therefore it is sufficient, as in the smooth case, to prove the existence of a non-zero section of $\pi _\ast (\omega _T)(2)$. We have $\pi _\ast (\omega _T)_{\vert_D} \cong {\omega _S}_{\vert_D}$, because $S$ is smooth in a neighborhood of the hyperplane section $D$. Since $S$ is smooth in a neighborhood of $D$,
we have the exact sequence
$$0 \to \pi _\ast (\omega _T)(1) \to \pi _\ast (\omega _T)(2) \to \omega _D(1)\to 0,$$ which gives $h^0(\pi _\ast (\omega _T)(2)) =7$.
\end{proof}

\begin{lemma}\label{lemma5.5}
We have $\deg (S) \ne 7$.
\end{lemma}
\begin{proof}
In the case $\deg (S) =2^3-1$ it is linked to a plane $M$ by the complete intersection $E$ of $3$ quadrics. But this case does not exist. Fix a general hyperplane $H\subset \PP^5$ and set $D:= S\cap H$ and $L:= M\cap H$. Recall that $S$ is cut out by quadrics; to get a contradiction it is sufficient to prove that every quadric containing $S$ contains $M$. So it is sufficient to prove that every quadric hypersurface of $H$ containing $S$ contains the line $L$. Therefore it is sufficient to prove that $\deg (L\cap D) \ge 3$.

Set $F:= E\cap H$. Since $F$ is the complete intersection of 3 quadrics of $H$, then $\omega _F \cong \Oo _F(1)$. In our set-up $S =\Phi$ has only isolated singularities and hence $D$ is a smooth curve; therefore $F$ has planar singularities and the degree of $L\cap D$ is obvious we only need $\deg (L\cap D)\ge 3$ and this is true for the following reason; we have $L\cap D\ne \emptyset$, because $\omega _L\ne \Oo _F(1)$; $L$ does not intersect transversally $D$ at one point (resp. two points), because $\omega _L(1) \ne \Oo _F(1)$ (resp. $\omega _L\ne \Oo _F(1)$). The case when $D\cap L$ is the unique point at which $L$ is an ordinary tangent of $D$ is excluded, because $\omega _L\ne \Oo _F(1)$.
\end{proof}

\subsection{Case of $s\geq 2$}$ $

Assume $s\ge 2$ and then we may assume $d\geq 17$. Let $E\subset \PP^5$ be the complete intersection of $3$ general quadric hypersurfaces containing $C$. Since $C$ is cut out in $\PP^5$ by some quadric hypersurfaces and a unique quartic hypersurface, we have $\dim (E)=2$, i.e. $E$ is a complete intersection surface and $\deg (E)=8$, not necessarily reduced.

Let $E_1,\dots ,E_z$ be the irreducible components of $E_{red}$, i.e. $z$ is the number of the irreducible components of $E_{red}$, and then we have $\sum _i \deg (E_i) \le 8$.  By rearranging the indices, we can take $E_1,\dots ,E_w$ contained in $\Psi$ for $0 \le w \le z$. Since $C$ is cut out in $\PP^5$ by some quadric hypersurfaces and a unique quartic hypersurface, we have
$$d\le \sum _{i=1}^{w} 4\deg (E_i) +\sum _{i=w+1}^{z} 2\deg (E_i).$$
Since $\omega _C \cong \Oo _C(2) \ne \Oo _C(3)$ and so $C$ is not the scheme-theoretic intersection of $E$ and a quartic hypersurface, so we have $z>w$. In particular we have $d \le 30$. If $\deg (E_j) =1$ for some $j$, then $E_j$ cannot contribute to $C$, because no plane curve $D$ of degree $2$ or $4$ has $\omega _D \cong \Oo _D(2)$ and so we have $d\le 28$.

Let $S_1,\dots ,S_k$ with $k>0$ be the positive dimensional irreducible components of the reduction of the base locus of $|\Ii _{C,\PP^5}(2)|$ (a priori the base scheme may have multiple components and isolated points). Since $C$ is cut out in $\PP^5$ by some quadric hypersurfaces and a unique quartic hypersurface, we have $1 \le \dim (S_i)\le 2$. Let $h$ be the number of components with $\dim (S _i)=2$. Call $U$ any quartic hypersurface such that $C = (S_1\cup \cdots \cup S_k)\cap U$. By reordering we may assume $\dim (S_i) = 2$ if and only if $i \le h$. Notice that if $i \le h$, then $S_i = E_j$ for some $j \le w$. We also have $w=h$. A component $S_i$ with $i>h$, i.e. $\dim (S_i)=1$, either disappear taking the intersection with a quadric hypersurface (i.e. its intersection with a quartic hypersurface is contained in some other $S_j)$ or it is a connected component of $C$. Since $d \ge 17$, so we may assume $h = w>0$. No $S_i$ may be a plane, because a quartic plane curve is canonically embedded. Fix an integer $i$ with $1\le i \le h$. The proof of Lemmas \ref{lemma5.3}, \ref{lemma5.4} and \ref{lemma5.5} work verbatim by taking $S_i\cap U_4$ instead of $C$ (we do not claim that $S_i\cap U_4$ is the union of all connected components of $C$, only of some of them). Hence we have $\deg (S_i)  \in \{2,3,4\}$.

\begin{lemma}
We have $\deg (S_i) \ne 3$.
\end{lemma}
\begin{proof}
Let $\deg (S_i ) =3$ for some $i\le h$. Since no cubic hypersurface of $\PP ^3$ is cut out by a single quartic hypersurface, the linear span $\langle S_i\rangle$ has at least dimension $4$. It implies $\dim \langle S_i\rangle =4$ and $S_i$ is a minimal degree surface of $\PP^4$. Thus $S_i$ is either a cone over a rational normal curve of $\PP^3$ or a smooth and complete embedding of $F_1$ by the linear system $|h+2f|$. Hence $A_i:= S_i\cap U_4 \in |4h+8f|$; any such curve is connected and so $A_i$ is a connected component of $C$. Since $\omega _{F_1} \cong \Oo _{F_1}(-2h-3f)$, we get $\omega _{A_i} \cong \Oo _{A_i}(2h+5f)$. So we have $\deg (A_i) = 12$ and
$$2p_a(A_i)-2 = (4h+8f)\cdot (2h+5f)= -8+20+16=28;$$
it is not admissible as genus. Now assume that $S_i$ is a cone with vertex $o$ and let $\pi: T\to S_i$ be the minimal desingularization. We have $T\cong F_3$ with $h = \pi ^{-1}(o)$ and $\pi$ is induced by the complete linear system $|h+3f|$. We have $A_i\in |4h+12f|$; again it is connected. Since $\omega _{F_3} \cong \Oo _{F_3}(-2h-5f)$ and $A_i$ does not pass through the vertex $o$ of $S_i$ we get $2p_a(A_i)-2 = (4h+12f)\cdot (2h+7f) = 28$, a contradiction.
\end{proof}

\begin{lemma}
We have $\deg (S_i) \ne 4$.
\end{lemma}
\begin{proof}
Let $\deg (S_i) =4$ for some $i\le h$. Setting $m:= \dim (\langle S_i)\rangle$, we have $m\ge 4$ since $S_i \subset E$ and $E \cap \langle S_i\rangle$ is cut out by quadrics. If $m=4$, we have $\omega _{S_i}\cong \Oo _{S_i}$, because $S_i\subset E\cap \langle S_i\rangle$ and $E \cap \langle S_i\rangle$ is the complete intersection of two quadric hypersurfaces of $\langle S_i\rangle$. Thus we have $\omega _{S_i\cap U_4} \cong \Oo _{S_i\cap U_4}(4)$, a contradiction.

Now assume $m=5$ and then $S_i$ is a surface of minimal degree in its linear span. Up to projective transformation $S_i$ is obtained in one of the following ways; in all cases $A_i = S_i\cap U_4$ is connected.

 \quad (a) Take $F_0$ embedded by the complete linear system $|h+2f|$ and $A_i\in |4h+8f|$; we have $\omega _{F_0} \cong \Oo _{F_0}(-2h-2f)$ and so $\omega _{A_i} \cong \Oo _{A_i}(2h+6f)$. We get $\deg (A_i) =16$ and $2p_a(A_i) -2 = 40 \ne 32$, a contradiction.

 \quad (b)  Take $F_2$ embedded by the complete linear system $|h+3f|$ and $A_i\in |4h+12f|$; we have $\omega _{F_2} \cong \Oo _{F_2}(-2h-4f)$ and so $\omega _{A_i} \cong \Oo _{A_i}(2h+8f)$. We get $\deg (A_i) =16$ and $2p_a(A_i) -2 = 40 \ne 32$, a contradiction.

 \quad (c) If $S_i$ is singular, then $S_i$ is a cone over a rational normal curve of $\PP^4$. Since
 $A_i =S_i\cap U_4$ scheme-theoretically and $A_i$ is smooth, so $A_i$ does not contain the vertex $o$ of $S_i$. If $\pi : T\to S_i$ is a minimal desingularization of $S_i$, then we have $T\cong F_4$ with $h = \pi ^{-1}(o)$ and $\pi $ is induced by the linear system $|h+4f|$. Note $\pi ^{-1}(A_i) \cong A_i$ and $\pi ^{-1}(A_i) \in |4h+16f|$. Since $\omega _{F_4} \cong \Oo _{F_4}(-2h-6f)$, we have $\omega _{\pi ^{-1}(A_i)} \cong \Oo _{\pi ^{-1}(A_i)}(2h+10f)$ and so $2p_a(A_i)-2 = 40 \ne 32$, a contradiction.
\end{proof}

\begin{lemma}
We have $\deg (S_i) \ne 2$, except when $k=h=2$ with $d=16$.
\end{lemma}
\begin{proof}
Let $\deg (S_i)=2$ for some $i\le h$ and so $S_i$ is an irreducible quadric surface. The complete intersection curve $A_i:= S_i\cap U_4$ is certainly admissible; it has $\omega_{A_i} \cong \Oo_{A_i}(2)$ (if there were no other components of $C$ it would give $\Oo _X(1)^{\oplus 2}$ as the associated bundle). Since $A_i$ is connected due to its ampleness in $S_i$, it is a connected component of degree $8$ of $C$. By the previous steps we have $\deg (S_i)=2$ for all $i\le h$. Note that the case $k= h=2$ is the one found with $d=16$.

We may write $C_i$ instead of $A_i$ for $i\le h$, because each $A_i$ is connected.

\quad (a) Assume $h=1$ and then $\Ii _{C_i,\PP^5}(2)$ is globally generated for all $i\ge 2$. If $\langle C_i\rangle \ne \PP^5$, then we are in the case $H^0(\Ee (-1)) \ne 0$. Take $E_i\subset E_{red}$ with $i\ge 2$ such that $C_i\subset E_{red}$. Since $C_i$ is contained in a complete intersection of $4$ quadrics, we have $d_i\le 16$ and $\langle C_i\rangle $ has neither dimension $3$ nor dimension $4$ (see the proof of Prop. \ref{prop4.5}). Since $\langle C_i\rangle =\PP^5$, the genus bound gives $d_i \ge 14$. However there is an irreducible component $E_j$ of $E_{red}$ such that $C_i$ is contained in the intersection of $E_j$ with a quadric hypersurface not containing $E_j$. In particular we have $d_i \le 2\deg (E_j)$. Since $h>0$ and $\deg (E)=8$, we have $\deg (E_j)\le 6$ and so $d_i\le 12$, a contradiction.

\quad (b) Assume $h\ge 2$. We have $\cup _{i\le h} S_i\subset E$ and the same quartic hypersurface $U_4$ must work for all $S_i$ with $\dim (S_i) =2$. If $\langle S_1\rangle =\langle S_2\rangle$, then we have $\dim (S_1\cap S_2) =1$ and so $C_1\cap C_2 = S_1\cap S_2\cap U_4\ne \emptyset$, a contradiction. Now assume that $\langle S_1\rangle \cap \langle S_2\rangle$ is a plane, i.e. $\dim (\langle S_1\cup S_2\rangle )=4$. Since $E\cap \langle S_1\cup S_2\rangle$ is scheme-theoretically cut out by quadrics, we get that $S_1\cup S_2$ is the complete intersection of two
quadric hypersurfaces. Thus we have $\omega _{S_1\cup S_2} \cong \Oo _{S_1\cup S_2}(-1)$ and so $\omega _{C_1\cup C_2} \cong \Oo _{C_1\cup C_2}(3)$, a contradiction. Therefore we have $\langle S_1\cup S_2\rangle =\PP^5$. Since $h=w$ and $2w <8$, we have $h \le 3$.

Assume $h=k=3$, i.e. $C = U_4\cap (S_1\cup S_2\cap S_3)$. This case gives a solution only if $\Ii _{S_1\cup S_2\cup S_3,\PP^5}(2)$ is spanned. For a general hyperplane
$H\subset \PP^5$, set $F:= H\cap H$ and $D_i:= H\cap S_i$ with $i=1,2,3$. Since $\emptyset = C_i\cap C_j =S_i\cap S_j\cap U_4$, we have $\dim (S_i\cap S_j)\le 0$ for all $i\ne j$ and so $D_1\cap D_2=D_1\cap D_3 = D_2\cap D_3=\emptyset$. It implies that any two planes $\langle D_i\rangle$ and $\langle D_i\rangle$, $i\ne j$, meet in a unique point, say $O_{ij}$. Each $D_i$ is a smooth conic and $D_1\cup D_2 \cup D_2\subset F$. Note that the curve $F$ is a complete intersection of $3$ quadric hypersurfaces of $H$ and in particular we have $h^0(\Oo _F)=1$. To get a contradiction it is sufficient to prove that $\Ii _{D_1\cup D_2\cup D_3,H}(2)$ is not globally generated. Since $D_1\cup D_2\cup D_3$ is not the complete intersection of $3$ quadric hypersurfaces, it is sufficient to prove that $h^0(H,\Ii _{D_1\cup D_2\cup D_3,H}(2))\le 3$ and it would be induced by showing $h^0(H,\Ii _{D_1\cup D_2\cup D_3\cup \{P_1\cup P_2\cup P_3\},H}(2))=0$ for three points $P_i\in \langle D_i\rangle \setminus D_i$ with $i=1,2,3$. Since $D_i$ is the only conic of $\langle D_i\rangle$ containing $D_i$, we have
$$h^0(H,\Ii _{D_1\cup D_2\cup D_3\cup \{P_1\cup P_2\cup P_3\},H}(2)) = h^0(H,\Ii _{\langle D_1\rangle\cup \langle D_2\rangle \cup \langle D_3\rangle,H}(2)).$$
Every quadric hypersurface  of $H$ containing $\langle D_i\rangle\cup \langle D_j\rangle$
is a cone with vertex $O_{ij}$ and so every element of $|\Ii _{D_1\cup D_2\cup D_3\cup \{P_1\cup P_2\cup P_3\},H}(2)|$ is a cone with
vertex containing $\{O_{12},O_{23},O_{13}\}$. Let $\ell : \PP^4\setminus \{O_{12}\}\to \PP^3$ be the linear projection from $O_{12}$. The set $J:= \ell (\langle D_3\rangle \setminus \{O_{1,2}\})$ is either a line if $O_{1,2}\in \langle D_3\rangle$ or a plane if $O_{1,2}\notin \langle D_3\rangle$. The sets $\ell (\langle D_1\rangle \setminus \{O_{1,2}\})$ and $\ell (\langle D_2\rangle \setminus \{O_{1,2}\})$ are disjoint lines, none of them containing $J$. Therefore, by taking the linear projection from $O_{12}$, we see that $h^0(H,\Ii _{\langle D_1\rangle\cup \langle D_2\rangle \cup \langle D_3\rangle,H}(2))>0$ if and only if $O_{12}\in \langle D_3\rangle$, i.e. $O_{12} =O_{13}=O_{23}$. Thus we may assume $O_{12}=O_{13}=O_{23}$. In this case we have
$$h^0(H,\Ii _{D_1\cup D_2\cup D_3\cup \{O_{12}\},H}(2))=h^0(H,\Ii _{\langle D_1\rangle\cup \langle D_2\rangle \cup \langle D_3\rangle,H}(2))= 1$$
since the set $\ell ((\langle D_1\rangle \cup \langle D_2\rangle \cup \langle D_3\rangle )\setminus \{O_{12}\})$ is the disjoint union of $3$ lines in $\PP^3$. Thus we have $h^0(H,\Ii _{D_1\cup D_2\cup D_3,H}(2)) \le 2$.

If $2\le h<k$, then we get $\langle C_k\rangle =\PP^5$ as in the beginning of (b), and so we have $d_k\ge 14$. It implies $d_k\le 2\deg (E_j) <14$, a contradiction.
\end{proof}


\section{CICY of degree $9$}

Assume that $X=X_{3,3}$. Let $U$ be a general cubic hypersurface in $\PP^5$ containing $X$; if necessary let $U'$ be another cubic hypersurface containing $X$ and so we have $X = U\cap U'$. Recall that $S_i$ is defined to be a fixed reduced and irreducible component $S_i \subseteq \Psi$ containing $C_i$. If $\dim (S_i) =1$, then $S_i =C_i$, because $S_i$ is irreducible and $S_i\supseteq C_i$. If $s=1$, then set $S:= S_1$.

If $\dim (\langle C_i \rangle)=3$, then $C_i$ is not the complete intersection of a quadric surface and a quartic surface $T$, since $T\not \subseteq X$. In particular, we have $d_i\geq 9$. Since $\langle C_i \rangle \cap X$ is not a surface and it contains $C_i$, we have $d_i=9$, i.e. $C_i$ is the complete intersection of two cubic surfaces. In other words, it is the intersection of $X$ with a linear subspace of $\PP^5$ with codimension $2$. Note that the case $s=2$ with $\dim (\langle C_i \rangle)=3$ for all $i$ is in Example \ref{b1} and the case $s=1$ with $\dim (\langle C_1 \rangle)=3$ is the case $H^0(\Ee(-1))\ne0$.

We recall that we write $\pi (d,n)$ for the upper bound on the genus for non-degenerate curves of degree $d$ in $\PP^n$. 
If $\dim (\langle C_i\rangle ) =4$, then we have $d_i \ge 11$, because $d_i=g_i-1$ and $\pi  (11,4) =12$. Similarly if $\dim (\langle C_i\rangle ) =5$, we have $d_i\ge 14$ since $\pi(14,5)=15$.

Our main goal in this section is to prove the following assertion.

\begin{proposition}
If $\Ee$ is a globally generated vector bundle of rank $2$ without trivial factors on $X_{3,3}$ with $c_1(\Ee)=2$ and $h^0(\Ee(-1))=0$, then $\Ee$ is one of the bundles in Example \ref{b1}, \ref{b2}, \ref{b3}, \ref{inc}, except when $c_2(\Ee)=16$.
\end{proposition}

\subsection{Case of $s=1$} $ $

Let us assume first that $C$ is connected and so we have $S=S_1$ as the irreducible component of $\Psi$ containing $C$ and we have $S\cap X=C$ as scheme.

\begin{proposition}\label{prop.6.1}
We have $\dim (S)=1$ or $2$. Moreover if $\dim (S)=1$, then $\Ee$ is as in Example \ref{b2}.
\end{proposition}
\begin{proof}
We have $\dim (S)\le 3$ and $\dim (S) =1$ if and only if $S = C$. We have $h^0(\PP^5,\Ii _{S,\PP^5}(2)) =h^0(\Ii _S(2))$. Since $\Ii _C(2)$ is globally generated and $X$ is the complete intersection of two cubic hypersurfaces,
$\Ii _{C,\Phi }(3)$ and $\Ii _{C,S}(3)$ are also globally generated.

\quad{(a)} Assume $\dim (S) =3$. Since $C$ spans $\PP^5$ and $S$ is contained in the complete intersection of two hypersurfaces, we have $\deg (S) = 3$ or $4$. If $\deg (S) =4$, then $S$ is the complete intersection of two quadrics and since $C = X\cap S$ as schemes we get $\omega _C \cong \Oo _C(4)$, a contradiction. Now assume $\deg (S) =3$ and so $d=27$. Since $C$ is a complete intersection of two cubic surfaces inside $S$ and $C$ is smooth, $S$ is smooth at each point of $C$. Since $C$ is the complete intersection of two ample divisors inside $S$, the singular locus
of $S$ has dimension $\le 1$. Let $\pi : S'\to S$ be a desingularization map and then we have $\omega _C = \omega _S(6)_{\vert_C}= \pi _\ast (\omega _{S'})(6)_{\vert_C}$. Since $C$ is a complete intersection of two ample divisors of $S$ and $C$ is smooth, to get a contradiction it is sufficient
to find a nonempty effective Weil divisor $D$ as the zero-locus of a section of $ \pi _\ast (\omega _{S'})(6)$. Let $H\subset \mathbb {P}^{5}$ be a general hyperplane. Since we have $h^1(S, \pi _\ast (\omega _{S'})(i)) =0$ for all $i>0$ by the Grauert-Riemenschneider's vanishing theorem, the following exact sequence
$$0 \to \pi _\ast (\omega _{S'})(5) \to \pi _\ast (\omega _{S'})(6) \to \pi _\ast (\omega _{S'})(6)_{\vert_{H\cap S}}\to 0$$
shows that it is sufficient to prove that $H^0(\pi _\ast (\omega _{S'})(6)_{\vert_{H\cap S}})$ has a nonzero section with non-zero locus. Let $H'$ be a general hyperplane of $H$. Since $S$ has at most $1$-dimensional singular locus, $H'\cap S$ is a smooth cubic surface. Applying the Grauert-Riemenschneider's vanishing theorem to $S\cap H$ we see that it is sufficient to prove that $h^0(H'\cap S, \pi _\ast (\omega _{S'})(6)_{\vert_{H\cap S}}) \ge 2$. We have $\pi _\ast (\omega _{S'})(6)_{\vert_{ H\cap S}} = \omega_{ H' \cap S}(4) \cong \Oo _{H'\cap S}(3)$.

\quad{(b)} Assume $\dim (S) =1$. Since $C$ is contained in the complete intersection of $4$ quadric hypersurfaces of $\PP^5$, we have $d \le 2^4 =16$ with equality only if it is a complete intersection of $4$ quadrics; this case is possible because it has $\omega _C \cong \Oo _C(2)$ and it is as in Example \ref{b2}. Now assume $d<16$ and then we follow the argument in Lemma \ref{lemma5.2}.
Let $E$ be the intersection of $4$ general quadric hypersurfaces containing $C$. Since $C$ is the scheme-theoretic intersection of quadric hypersurfaces, $E$ is a curve, $E$ is the complete intersection of
$4$ hypersurfaces of degree $2$ and $E = C\cup F$ with $F$ a scheme of dimension $1$ with $\deg (F) =16-d$, i.e. $F$ is a line if $d=15$, while $F$ is either a double structure
on a line or the union of two disjoint lines or a plane conic if $d=14$. Since $E$ is a complete intersection, the adjunction formula gives $\omega _E \cong \Oo _E(2)$ and so ${\omega _E}_{\vert_C} \cong \Oo _C(2)$. Since $E$ is a complete intersection, the cohomology groups of the sheaves $\Ii _{E,\PP^5}(t)$ give $h^0(\Oo _E)=1$. Therefore $F\cap C \ne \emptyset$ and so $\omega _C \ne {\omega _E}_{\vert_C} \cong \Oo _C(2)$, a contradiction.
\end{proof}

Assume now $\dim (S) =2$. Since $S$ is contained in a complete intersection of quadrics, we have $\deg (S) \le 8$ and the equality holds if and only if $S$ is a complete intersection of $3$ quadrics. Since $C$ is the intersection of $S$ and the intersection of two cubics, we have $\deg ( C) \le 3\deg (S)$. Since $d \ge 13$ (when $s=1$ we have $d\ge 14$), we have $\deg (S)\ge 5$. By the Bertini theorem the scheme $S\cap U$ is smooth outside $C$. Since $2d > 3\deg (S)$, we get that $S\cap U$ does not contain $C$ with multiplicity $\ge 2$. Since $S$ is reduced and $\dim (S)\ge 1$ at each $P\in S$, the local ring $\Oo _{S,P}$ is reduced. We get that the scheme $U\cap S$ is a reduced curve.

\begin{lemma}\label{lemma6.2}
If $\deg (S)=8$, then we have $\deg (C)=18$ and its associated bundle exists on some $X_{3,3}$.
\end{lemma}
\begin{proof}
Since $\deg (S) =8$, we have $\omega _S \cong  \Oo _S$. We have $d\ne 24$, because $d=24$ would imply $C = U\cap S$, as schemes, and so $\omega _C\cong \Oo _C(3)$. Therefore $U\cap S = C\cup D$ with $D$ a reduced curve of degree $24-d>0$, smooth outside $C\cap D$.
We have ${\omega _{U\cap S}}_{\vert_C} \cong \omega _C(C\cap D)$ and so $\deg (C\cap D) = d$. However, since $C$ is cut out scheme-theoretically
by $U$ and $U'$ inside $S$ and $\deg (U'\cap D) =3(24-d)$, we get $d= 3(24-d)$ and so $d=18$. This case arises on some $X_{3,3}$ by Example \ref{inc}.
\end{proof}

\begin{lemma}\label{lemma6.3}
We have $\deg (S)\ne 7$.
\end{lemma}
\begin{proof}
There are $3$ quadric hypersurfaces $Q, Q',Q''\subset \PP^5$ such that $Q\cap Q'\cap Q'' =S\cup M$ with $M$ a plane. Let $V\subset \PP^5$ be a general $3$-dimensional subspace and then $V\cap (S\cup M)$ is the union of $8$ points and it is the complete intersection of $3$ quadric surfaces of $V$. Thus we have $h^0(V,\Ii _{V\cap (S\cup M)}(2)) =3$. Since no three points of $S\cap V$ are collinear and no six of them are coplanar, we have $h^1(V,\Ii _{S\cap V}(2)) =0$. We get $h^0(V,\Ii _{V\cap S}(2)) =3$ and so any quadric surface containing $V\cap S$ contains the point $V\cap M$; this is a particular case of the Cayley-Bacharach property for complete intersections (see \cite{egh}). Therefore every quadric hypersurface of $\PP^5$ containing $S$ contains $M$. Hence the base-locus of $\Ii _C(2)$ contains the space curve $M\cap X$, a contradiction.
\end{proof}

\begin{lemma}\label{lemma6.5}
The case with $\deg (S)=6$ arises only as in Example \ref{b3}, \ref{inc}, except when $d=16$.
\end{lemma}
\begin{proof}
Assume $\deg (S) =6$ and so we have $14\le d\le 18$. If $d=14$, then in the set-up of \cite[Theorem 3.15]{he} we have $(m_1, \epsilon, \mu)=(2,3,0)$ and so $\pi _1(14,5) =11 < g-1$. Thus $C$ is contained in a surface $T$ of degree at most $4$ by \cite[Theorem 3.15 (i)]{he}. Since $C$ is cut out by cubics and quadrics, we get $d\le 3\deg (T) <14$, a contradiction.

Now assume $d=15$ and then we have $\pi _1(15,5) = 16$. By \cite[Theorem 3.15]{he}, $C$ is contained in a surface $T\subset \PP^5$ with $\deg (T) \le 5$. Since $C$ is scheme-theoretically cut out by quadrics and cubics in $\PP^5$, we get $\deg (T) = 5$ and that $C$ is the complete intersection of $T$ and a cubic hypersurface $T'$. Since $C$ is smooth and $T\cap T' =C$ as schemes, $T$ is smooth in a neighborhood of $C$. Since $C$ is an ample divisor, $T$ has only finitely many singular points. In other words, $T$ is a weak Del Pezzo surface of degree $5$ with $\omega_T \cong \Oo_T(-1)$, giving $C$ with $\omega_C \cong \Oo_C(2)$. Such examples exist on some $X_{3,3}$ as in Example \ref{b3}.

Assume $d=17$. For a general hyperplane $H\subset \PP^5$, set $D:= S\cap H$. Since $S$ has only isolated singularities, $D$ is a smooth and connected curve of degree $6$ spanning $H= \PP^4$. Thus we have $p_a(D)\le 2$. Since $D$ is cut out by quadrics in $H$, there is no line $T\subset H$ with $\deg (T\cap D) \ge 3$. The Berzolari formula for the number of trisecant lines to $D$ in \cite{ho}, \cite{bc} gives $p_a(D) =2$. For a general $U\in |\Ii _{C,\PP^5}(3)|$. We have $S\cap U = C\cup L$ with $L$ a line. Varying $U$ we see that $S$ contains infinitely many lines, i.e. it is the image of a base point free linear system on a $\PP^1$-bundle on $D$. Now for a vector bundle $\Ff$ of rank $2$ on $D$. Set $S':= \PP (\Ff)$. We know that $S$ comes from a complete linear system on $S'$. We use $\sim$ for numerical equivalence on $S'$. Since $D$ has positive genus, two numerically equivalent line bundles may have different cohomology group. We have $\Pic (S')/\sim\cong \ZZ^{\oplus 2} =\ZZ\langle h,f\rangle$, where $f$ is a fiber of the ruling $u: S'\to D$ of $S'$ and $h$ is a section of $u$ with minimal self-intersection. Set $e = -h^2$. By a theorem of C. Segre and M. Nagata in \cite[Ex. V.2.5]{h0}, we have $e\ge -2$. Since $D$ has genus $2$, the adjunction formula applied first to $h\cong D$ and then to $f \cong \PP^1$ gives $\omega _{S'} \sim -2h+(-e+2)f$. Since $\deg (S) =6$, the linear system mapping $S'$ onto $S$ is numerically equivalent to $h +(3+(e/2))f$; in particular $e$ is even. Since $\deg (C\cap L)=3$, the linear system whose member has $C$
as its isomorphic image is numerically equivalent to $3h+af$ for some $a\in \mathbb {Z}$. Since $17 =d= (3h+af)\cdot (h+ (3+(e/2)))f$, we have $a = 8+3(e/2)$. The adjunction formula gives
\begin{align*}
34 = 2g-2 &= (3h+(8+3(e/2))f)(h+(10+(e/2))f)\\& = -3e +30 +(3e/2) +8 +(3e/2),
\end{align*}
which is absurd.

The case $d=18$ is realized on some $X_{3,3}$ by Example \ref{inc}.
\end{proof}

\begin{lemma}
If $\deg (S)= 5$, then we have $\deg (C)=15$ and its associated bundle exists on some $X_{3,3}$ as in Example \ref{b3}.
\end{lemma}
\begin{proof}
Note that $d\ge 14$, because $\pi _5(14)=15$. First assume $d=15$. In this case $C$ is the complete intersection of $S$ and a cubic hypersurface and hence $S$ is smooth at each point of $S$ and with finitely many singular points. Example exists on some $X_{3,3}$ as in Example \ref{b3}. Now assume $d=14$. If $U\subset \PP^3$ is a general hypersurface of degree $3$ containing $X$, then we get $U\cap S = C\cup L$ with $L$ a line. Take a cubic hypersurface $U'\ne U$ with $U'\supset X$. Since $U\cap U'\cap S =X\cap S =C$, we get that the $0$-dimensional scheme $L\cap U'$ of degree $3$ is contained in $C$. Therefore $p_a(C\cup L) \ge p_a+2$. Since $p_a(C\cup L) = 16$, we get $g \le 14$, a contradiction.
\end{proof}

\subsection{Case of $s\ge 2$} $ $

Let us assume that $C$ is not connected.

\begin{lemma}\label{lemma6.7}
We have $S_i \ne S_j$ for some $i\ne j$.
\end{lemma}
\begin{proof}
Assume $S_i=S_j$ for all $i\ne j$ and let $S=S_i$. We have $S\cap X = C$ as schemes and $S$ is irreducible. Note that $\dim (S)\le 3$ since $S\cap X=C$ is a curve and $X$ is a complete intersection. Moreover we have $\dim (S)>1$ since $C$ is not connected. If $\dim (S)=3$, then $C$ is the complete intersection of $3$ hypersurfaces of $\PP^5$ and so it is connected, a contradiction. Thus we may assume $\dim (S)=2$.

Since $S$ is a surface cut out by quadrics in $\PP^5$, so $\deg (S)\le 8$ and the equality holds if and only if $S$ is a complete intersection of $3$ quadric surfaces. Thus we have $d\le 28$. It implies $s=2$, except when $s=3$ and $\dim (\langle C_i\rangle )=3$ for all $i$ and $\deg (S) =8$. But this exception does not arise because each quadric hypersurface of $\PP^5$ containing one such $C_i$ also contains the $3$-dimensional linear subspace $\langle C_i\rangle$. We also get $d \le 3\deg (S)$. Since $s\ge 2$, we have $\deg ( C) \ge 18$. Since $C\subseteq S\cap U$, then $\deg (S) \ge 6$.

If $\deg (S)=8$, then we have $d=18$ as in the proof of Lemma \ref{lemma6.2}. Since $s\ge 2$, we get $s=2$, $d_1=d_2 =9$ and so $S_i =\langle C_i\rangle \cong \PP^3$, a contradiction. Assume $\deg (S)=7$. Since $d \le 21$, we get $s=2$ and at least one $C_i$, say $C_1$, has degree $9$. In this case we have $S =\langle C_1\rangle \cong \PP^3$, a contradiction. If $\deg (S)=6$, we get a contradiction similarly because $d\le 18$.
\end{proof}

\begin{lemma}\label{a}
We have $\dim (S_i)\ge 2$ for all $i$.
\end{lemma}
\begin{proof}
Assume the existence of an index $i$ with $\dim (S_i) = 1$, i.e. $S_i = C_i$. If $\dim (\langle C_i\rangle ) =3$, then we have $S_i =\langle C_i\rangle$, a contradiction. If $\dim (\langle C_i\rangle ) =4$, then we get $d_i \le 8$, because $S_i=C_i$ is cut out by quadrics inside $\langle C_i\rangle =\PP^4$, a contradiction. Now assume $\dim (\langle C_i\rangle ) =5$ and then we have $d_i \ge 14$. Let $E_i\subset \PP^5$ be the intersection of $4$ general quadric hypersurfaces containing $C_i$. Since $C_i$ is cut out by quadrics, so $E_i$ is a curve of degree $16$ containing $C_i$ and so we have $14 \le d_i \le 16$ . Notice that $d_i=16$ if and only if $C_i =E_i$. By the adjunction formula we have $\omega _{C_i} = \Oo _{C_i}(2)$.

If $d_i =16$, then $\Psi \ne E_i$ and so $h^0(\Ii _C(2)) \le 3$. Since $\Psi$ is cut out by at most $3$ quadric equations, each irreducible component of $\Psi _{red}$ has dimension at least $2$. It implies $\dim (S_i)\ge 2$, a contradiction. If $d_i=15$, then we have $E_i = C_i\cup L$ with $L$ a line. Since $\omega _{E_i} \cong \Oo _{E_i}(2)$ and $\omega _{C_i} \cong \Oo _{C_i}(2)$, we get $L\cap C_i= \emptyset$, contradicting the connectedness of any complete intersection curve. If $d_i=14$, then we have $E_i = C_i\cup D$ with $D$ either a smooth conic or the disjoint union of two lines or a double structure over a line. We can get a contradiction as in the case $d_i=15$; we may also use \cite[Theorem 3.15]{he} to get that $C_i$ is contained in an irreducible surface of degree at most $5$; since $C_i\subset E_i$ and $E_i$ are cut out by quadrics, we would get $d_i\le 10$, a contradiction.
\end{proof}

Assume for the moment $\dim (S_1)=3$. Since $S_1\cap X$ is either $C_1$ or $C_1\cup C_3$ as a scheme, we have $d-d_2 \ge 9\deg (S_1)$. Since $d_2\ge 9$ and $d<36$, we get $\deg (S_1) \le 2$. Assume for the moment $\deg (S_1) =2$. We get that $S_1$ is a quadric hypersurface of a $4$-dimensional space. and hence $\omega _{S_1} \cong \Oo _{S_1}(-3)$. Hence $\omega _{S_1\cap X} \cong \Oo _{S_1\cap X}(3)$. Since
$C_1$ is a connected component of $S_1\cap X$, we get a contradiction.
Hence all 3-dimensional $S_i$ are linear spaces. So we get that $\dim (S_i) =3$ if and only if $d_i=9$ and $C_i$ is the complete intersection of $X$ with a codimension 2 linear
subspace of $\PP^5$. Now fix $i\in \{1,\dots ,s\}$ with $\dim (S_i)=2$ (if any). Since $\deg (C_i)\ge 11$ and $C_i \subseteq S_i\cap U$, we get $\deg (S_i) \ge 4$; if $\dim (\langle C_i\rangle )=5$, then we also get $\deg (S_i) \ge 5$.

\begin{lemma}\label{b}
There exists an index $i$ with $\dim (S_i)=3$.
\end{lemma}
\begin{proof}
Assume $\dim (S_i) =2$ for all $i$. By the previous paragraph and the assumption $S_i\ne S_j$ for some $i, j$, we get that among the surfaces $S_i$, $1\le i \le s$, there are exactly two different surfaces, say $S_1$ and $S_2$. Since $S_1\cup S_2$ is contained in the intersection $E$ of $3$ quadric hypersurfaces with $\dim (E)=2$, then $\deg (S_1)=\deg (S_2) =4$. By Lemma \ref{a} we also get $\dim (\langle C_i\rangle ) = 4$ for all $i$. Since $S_1\cap U$ has degree $12$, it cannot contain two components of $C$. Thus we have $s=2$ and $d_i\in \{11,12\}$ for all $i$. Since $\dim (\langle C_i\rangle ) =4$, $\deg (S_i) =4$ and $(S_1\cup S_2)\cap \langle C_i\rangle$
is cut out by quadric hypersurfaces in $\langle C_i\rangle$, we get that each $S_i$ is a complete intersection of two quadric hypersurfaces in $\langle C_i\rangle$.

Assume for the moment $d_i=11$. Since $C_i \subset S_i\cap U$ for a general cubic hypersurface $U$ containing $X$, we get $S_i\cap U = C_i\cup L$ with $L$ a line. Since $S_i$ is a complete intersection of two quadrics in $\PP^4$, then $\omega _{S_i}\cong \Oo _{S_i}(-1)$ and hence $\omega _{S_i\cap U_i} \cong \Oo _{S_i\cap U}(2)$. Since $S_i\cap U$ is a complete intersection, it is connected and $p_a(C_i\cup L) = 13$. Since $g_i =12$, we get $\deg (C_i\cap L) =2$ and hence either $C_i\cup L$ is nodal or it has a unique singular point, which is an ordinary tacnode. Since ${\omega _{C_i\cup U}}_{\vert_L} \cong \Oo_L(2)$ and $\deg (L\cap C_i) =2$, we get a contradiction.

Therefore $d_1=d_2=12$ and each $C_i$ is the complete intersection of two quadric hypersurface of $\PP^5$, a cubic hypersurface and a hyperplane. Let $U\subset \PP^5$ be a general cubic hypersurface containing $C$ and then we have $C =U\cap (S_1\cup S_2)$. Since $s\ge 2$, to get a contradiction it is sufficient to prove that $S_1\cup S_2$ is the complete intersection of $3$ quadric hypersurfaces. Since $h^0(\PP^5,\Ii _{S_1\cup S_2}(2)) =h^0(\Ii _C(2)) \ge 3$ and $S_1\cup S_2$ is a surface of degree $8$, it is sufficient to prove that $h^0(\PP^5,\Ii _{S_1\cup S_2}(2)) \le 3$. Set $H_i:= \langle C_i\rangle$ and $M:= H_1\cap H_2$. Since $C$ spans $\PP^5$ and $s=2$, we have $H_1\ne H_2$. Therefore $M$ is a hyperplane of each $H_i$. Fix a general $A \subset M$ with $\sharp (A)=2$.
Since $S_i$ is a complete intersection of two quadric hypersurfaces inside $H_i$, we have $h^0(H_i,\Ii _{S_i,H_i}(2)) =2$ and $h^0(H_i ,\Ii _{S_i\cup M,H_i}(2)) =0$. Since $\sharp (A) =2$ and $A$ is general in $M$, so $h^0(H_i ,\Ii _{S_i\cup A,H_i }(2)) =0$. Therefore we have $h^0(H_1\cup H_2,\Ii _{S_1\cup S_2,H_1\cup H_2}(2)) \le 2$ and so $h^0(\PP^5,\Ii _{S_1\cup S_2}(2)) \le 3$.
\end{proof}

By Lemmas  \ref{a} and \ref{b} we may assume $2\le \dim (S_i)\le 3$ for all $i$ and the existence of one index $j$ with $\dim (S_j) = 3$. Without loss of generality we assume $\dim (S_1)=3$ and so $S_1 =\langle C_1\rangle$ is a $3$-dimensional linear subspace, $d_i=9$, and $C_1$ is the complete intersection of $\langle C_1\rangle$ and two cubic hypersurfaces.

\begin{lemma}\label{c}
We have $\dim (S_i)=3$ for all $i$.
\end{lemma}
\begin{proof}
Without loss of generality let us assume $\dim (S_2) = 2$ and then we have $\dim (\langle C_2\rangle )\ge 4$.

First assume $\dim (\langle C_2\rangle )=4$. Since the irreducible surface $S_2$ is cut out by quadric hypersurfaces inside $\langle C_2\rangle$, we have $\deg (S_2) \le 4$.
Since $d_2\ge 11$ and $C_2$ is contained in the intersection of $S_2$ with a cubic hypersurface $U$, we get $\deg (S_2) =4$. We also get that $S_2$ is the complete intersection of two quadrics of $\langle S_2\rangle$ and so $d_2\in \{11,12\}$. If $d_2=11$, then $C_2$ is contained in a surface $T$ of degree at most $3$ by \cite[Theorem 3.15 (i)]{he} with $(d,r,m_1, \epsilon_1, \mu)=(11,4,2,2,0)$. Since $C_2$ is cut out by quadric and cubics, we get $d_2\le 3\deg (T)\le 9$, a contradiction. Now assume $d_2=12$ and so $C_2 = S_4\cap U$. Since $h^0(H,\Ii _{S_2}(2)) =2$ and $\langle C_1\rangle \cap \langle C_2\rangle$ is at least one plane, the set $H\cap \Psi$ contains a quadric hypersurface of $H$, a contradiction.

Now assume $\dim (\langle C_2\rangle)=5$ and so $d_2\ge 14$. As in (b) of the proof in Proposition \ref{prop.6.1}, we can get  $d_2\ne 14$. Since $d_2\le 3\deg (S_2)$, we also get $5 \le \deg (S_2) \le 8$. Since $S_2$ spans $\PP^5$, no reducible quadric contains $C_1\cup C_2$. Since $s\ge 2$, $C$ is not the complete intersection of $X$ with two  quadric hypersurfaces. Thus we have $h^0(\PP^5,\Ii _{S_1\cup S_2}(2)) = h^0(\PP^5,\Ii _{C_1\cup C_2}(2)) \ge 3$.

If $\deg (S_2)=8$, then we have $h^0(\PP^5,\Ii _{S_2}(2)) =3$. Since $\langle C_1\rangle \nsubseteq S_2$, so $h^0(\Ii _{C_1\cup C_2}(2)) \le 2$. It is not possible since $C$ is not connected and so not the complete intersection
of two quadrics and two cubics, which implies $h^0(\Ii _{C_1\cup C_2}(2))>2$. Similarly as in Lemma \ref{lemma6.3} we have $\deg (S_2)\ne 7$.

Assume $\deg (S_2) =5$. Since $d_2\ge 15$, we get that $C_2$ is the complete intersection of $S_2$ and a cubic hypersurface. As in the proof of Lemma \ref{lemma6.5} we have bundles as in Example \ref{b3}, but its associated curve $C$ must be connected, a contradiction.

Now assume $\deg (S_2)=6$. Since $C_2 \subseteq S_2\cap U$ with $U$ a cubic hypersurface, we have $15\le d_2\le 18$. As in Lemma \ref{lemma6.5} we exclude the case $d_2=17$. If $d_2=15$ and so $\pi _1(5,15)=16$, then $C_2$ is contained in a surfaces of degree $\le 5$ by \cite[Theorem 3.15]{he} and so it is a complete intersection of a cubic surface and a surface of degree $5$. As in the case of $\deg (S_2)=5$ above, we get a contradiction.

\quad (a) Assume $d_2=16$. Since $h^0(\Oo _{C_2}(2)) =g_2 =17$ and $h^0(\Oo _{\PP^5}(2)) =21$, we have $h^0(\PP^5,\Ii _{C_2,\PP^5}(2)) \ge 4$. Since $d_2 > 2\deg (S_2)$, every quadric hypersurface containing $C_2$ also contains $S_2$. For any irreducible and nondegenerate curve $T\subset \PP^5$, let $\Sec (T)$ denote its secant variety, i.e. the closure in $\PP^5$ of the union of all lines $L\subset \PP^5$ containing at least two points of $T$ and then we have $\dim (\Sec (T)) =3$ by \cite[Remark 1.6]{a}. Thus we have $\langle C_1\rangle \cap \Sec(C_2)\ge 1$. Let $M$ be a positive-dimensional irreducible component of $\langle C_1\rangle \cap \Sec(C_2)$. Since $C_2\nsubseteq \langle C_1\rangle$, so $\langle C_1\rangle \cap C_2$ is finite. Fix $P\in M$ with $P\notin C_2$. Since
$C_2$ is a smooth curve, $\Sec (C_2)$ is the union of all lines $L\subset \PP^5$ with $\deg (C_2\cap L) \ge 2$, not just the closure of this union. Therefore there is a line $L$ containing $P$ and with $\deg (L\cap (C_2\cup \{P\}) \ge 3$. Since $P\in \langle C_1\rangle \subset \Psi$, we get $L\subset \Psi$. Taking the closure of the union of all these lines $L$ for some $P\in M\setminus (M\cap C_2)$ we get an irreducible variety $N\subset \Psi$ containing $C_2$. Since $\deg (C_2) >9$, we get that $\dim (N) =2$ and so $\deg (N)\ge 5$. Assume for the moment $N\subsetneq S_2$. Since $C_1\cup C_2$ is cut out scheme-theoretically by cubics, we have $d_2\le 3\cdot (\deg (N)+\deg (S_2))$ (e.g. by \cite[Theorem 2.2.5]{fov}), a contradiction. Thus we have $N =S_2$, and in particular $\dim (M)=1$ and $S_2$ is ruled by lines with $M$ meeting a general line of the ruling of $S_2$ at a unique point.
Since $S_2\cap U =C_2$, we also get the existence of an integer $c$ with $2\le c\le 3$ such that $\sharp (L\cap C_2) =c$ for a general line of the ruling of $S_2$. Let $H\subset \PP^5$ be a general
hyperplane and then the scheme $S_2\cap H$ is an integral curve of degree $6$ spanning $H$. Therefore the normalization $Y$ of $S_2\cap H$ has genus $q\le 2$. Since $S_2$ is ruled by lines, there is a $\PP^1$-bundle $\pi : \PP (\Ff )\to Y$ over $Y$ with $\Ff$ a spanned vector bundle of rank $2$ on $Y$ and $\Ff  \cong \pi _{\ast }(\Oo _{\PP (\Ff )}(1))$ and a morphism $u: \PP (\Ff )\to \PP^5$ with image $S_2$, birational onto its image and with its fiber of $\pi$ mapped onto a line of the ruling of $S_2$. Since $\omega _{C_2}$ is very ample, so $C_2$ is not hyperelliptic. Since $\omega _{C_2} \cong \Oo _{C_2}(1)$ and $\Oo _{C_2}(1)$ is very ample, the canonical model of $C_2$ has no trisecant line. Therefore $C_2$ is not trigonal. Note that we have $q>0$ since $c\in \{2,3\}$. Let $C'$ denote the only irreducible curve of $\PP(\Ff )$ with $u(C') =C_2$. Note that $\Pic (\PP (\Ff ))/\sim \cong \ZZ^{\oplus 2}=\ZZ\langle h,f\rangle$ with a fiber $f$ of $u$ and a section $h$ of $ui$ with minimal self-intersection. Set $e = -h^2$ and then we have $e\ge -q$ by a theorem of C. Segre and N. Nagata in \cite[Ex. V.2.5 (d)]{h0}. Since $\deg (S_2)$ is even, so $e$ is even. Since $\Ff$ is spanned, we have $e \le 6$ with the equality if and only if $S_2$ is a cone. We have $\omega _{\PP (\Ff )}\sim -2h+(2q-2-e)f$ and $\Oo _{\PP (\Ff )}(1) \sim h+af$ for some $a\in \ZZ$. Since $\deg (S_2) =6$ and $h^2=-e$, then we have $a = 3 +(e/2)$.

\quad (a1) Assume $c=3$. We have $C_2\sim 3h+bf$ for some $b$. Since $\deg (C_2) = 16$, then $(3h+bf)\cdot (h+(3+(e/2))f) = 16$, i.e. $b = 16 +(3e)/2$. Since $\omega _{\PP (\Ff )} \sim -2h+(2q-2 -e)f$, the adjunction formula gives
\begin{align*}
32 = 2g_2-2 &= (3h +(16+3e/2)f)\cdot (h+(2q-2+16-(e/2)f)\\
& = -3e + 16+3e/2 +6q-6 +48-3e/2.
\end{align*}
Since $e$ is even and $-6 \le e \le 2$, we get a contradiction.

\quad (a2) Now assume $c=2$. Writing $C_2\sim 2x+bf$, we get $(2h+bf)\cdot (h+(3+e/2)f) = 16$, i.e. $b = 16+e$ and so $32 =2g_2+2 = (2h+(16+e)f)\cdot ((2q-2+16)f) = 32 +2q-2$. It implies $q=1$ and so $C_2$ is a bielliptic curve. Since $\Oo _{C_2}(1) = \omega _{C_2}(-1)$,
both $\Oo _{C_2}(-1)$ and $\Oo _{C_2}(1) = \omega _{C_2}(-1)$ are spanned, i.e. $\Oo _{C_2}(1)$ is a primitive line bundle in the sense of \cite{ckm}. The Clifford index $\mathrm{Cliff}(\Ll)$ of a special line bundle $\Ll$ on $C_2$ is the integer $\deg (\Ll)+2 -2h^0(\Ll)$. We have $\mathrm{Cliff}(\Oo _{C_2}(1)) = 16+2-2h^0(\Oo _{C_2}(1)) \le 12$. Since $C_2$ is bielliptic, it is not trigonal or hyperelliptic, but it has infinitely many $g^1_4$'s, i.e. line bundles $\Ll$ with $\mathrm{Cliff} (\Ll) =2$. By \cite[2.2.1, 2.2.3 or 2.3.1]{ckm} each primitive line bundle $\Ll$ on $C_2$ has either Clifford index $2$ or $g-3 = 14$. Therefore we have $\mathrm{Cliff}(\Oo _{C_2}(1)) =2$, i.e. $h^0(\Oo _{C_2}(1)) = 8$, contradicting the Castelnuovo's genus bound in $\PP^7$, since $\pi (16,7) = 12$.

\quad (b) Assume $d_2 =18$. Since $C_2\subseteq S_2\cap U$, we get that $C_2$ is the complete intersection of $S_2$ and a cubic hypersurface. Since $C_2$ is smooth and an ample Cartier divisor of $S_2$, we get that $S_2$ is smooth at each point of $C_2$ and it is has only finitely many singular points. Look at $\Sec (C_2)$ and its intersection with $\langle C_1\rangle$. We get that $S_2$ is ruled by lines and that a general line of this ruling meets $C_2$ at exactly $3$ points. Take $\PP (\Ff)$, $u$, $C'$ as in (a) with $\omega _{\PP (\Ff )}\sim -2h+(2q-2-e)f$ and $\Oo _{\PP (\Ff )}(1) \sim h+(3+(e/2))f$. Since $C_2$ is the complete intersection of $S_2$ and a cubic hypersurface, we have $C'\sim 3h+(9+(3e/2))f$ and so we have
\begin{align*}
36 = 2g_2-2 &= (3h+(9+(3e/2))f)\cdot (h+(2q+7-e/2)f)\\
&= -3e+9+(3e/2)+2q+7 -e/2 =2q+16-e.
\end{align*}
Since $e\ge -q$ and $q\le 2$, we get a contradiction.
\end{proof}

\begin{remark}
We have $\langle C_i\rangle \ne \langle C_j\rangle$ for all $i\ne j$ since $C_i= X\cap \langle C_i\rangle$.
\end{remark}

\begin{proposition}
Let $\Ee$ be a globally generated vector bundle of rank $2$ without trivial factors on $X_{3,3}$ with $c_1(\Ee)=2$ and $h^0(\Ee(-1))=0$. If its associated curve $C$ is not connected, then $\Ee$ is as in Example \ref{b1}.
\end{proposition}
\begin{proof}
By Lemma \ref{c} we may assume $\dim (S_i)=3$ for all $i$. Without loss of generality assume $\dim (\langle C_1\rangle \cap \langle C_2\rangle ) =2$, i.e. with $H:= \langle C_1\cup C_2\rangle $ a hyperplane.  In case $s=2$ we have $H^0(\Ee (-1)) \ne 0$ and so we may assume $s=3$. The base locus of the linear system $|\Ii _{C_1\cup C_2,H}(2)|$ is the reducible quadric $\langle C_1\rangle \cup \langle C_2\rangle$. Thus $X\cap H$ is cut out by quadrics and two cubic hypersurfaces of $H$ if and only if $C_3\cap H \subset \langle C_1\rangle \cup \langle C_2\rangle$. It is not possible, because $C_3$ meets the hyperplane $H$, $C_3\cap C_1 =C_3\cap C_2 =\emptyset$ and $X\cap \langle C_i\rangle =C_i$, $i=1,2$.

Thus we have $\dim (\langle C_i\rangle \cap \langle C_j\rangle ) =1$ for all $i\ne j$, i.e. $\langle C_i\cup C_j\rangle =\PP^5$. The case $s=2$ is in Example \ref{b1}; we obtained the additional information that when $s=2$ the only bundles are the ones described in Example \ref{b1}. Now assume $s=3$. For all $i\ne j$, set $L_{ij} := \langle C_i\rangle \cap \langle C_j\rangle$ and $W_{ij}:= \langle C_i\rangle \cup \langle C_j\rangle$. The scheme $W_{ij}$ has $5$-dimensional Zariski tangent space at each point of $L_{ij}$. Therefore every quadric hypersurface containing $W_{ij}$ is a quadric cone with vertex containing $L_{ij}$. Therefore every quadric hypersurface containing $C_1\cup C_2\cup C_3$ is a quadric cone with vertex containing $L:= \langle L_{12}\cup L_{13}\cup L_{23}\rangle$. If $\dim (L) \ge 2$, then we have $L\cap X\ne \emptyset$ and so $C$ has at least one singular point, a contradiction. Thus we have $\langle C_1\rangle \cap \langle C_2\rangle \cap \langle C_3\rangle = L_{12}$. Let $V\subset \PP^5$ be a $3$-dimensional linear space with $L_{12}\cap V=\emptyset$. The scheme $V\cap (\langle C_1\rangle \cup \langle C_2\rangle \cup \langle C_3\rangle )$
is the union of $3$ disjoint lines. Since every element of $|\Ii _{C_1\cup C_2\cup C_3}(2)| = |\Ii _{\langle C_1\rangle\cup \langle C_2\rangle \cup \langle C_3\rangle}(2)|$ is a cone over the unique quadric surface containing $V\cap( \langle C_1\rangle\cup \langle C_2\rangle \cup \langle C_3\rangle )$, we get $h^0(\Ii _C(2)) =1$ and so $\Ii _C(2)$ is not globally generated.
\end{proof}


\section{Example on CICY of codimension $3$}
In this section we suggest an example of globally generated vector bundles $\Ee$ of rank $2$ on a CICY threefold $X$ of codimension $3$ with $c_1(\Ee)=2$ and $h^0(\Ee(-1))=0$, i.e. $X=X_{2,2,3}$. We do not know other examples, while our tools are not enough to say that they are the only ones. In \cite[Theorem 1.2]{Knutsen} the existence of some isolated and smooth curves is shown on the general CICY threefold.

\begin{lemma}\label{d1}
Let $S\subset \PP^6$ be a smooth surface of degree $6$ with $\omega _S \cong \Oo _S(-1)$, i.e. a Del Pezzo surface of degree $6$. Then $\Ii _{S,\PP^6}(2)$ is spanned. In other words, $S$ is scheme-theoretically cut out by quadrics.
\end{lemma}

\begin{proof}
Fix a hyperplane $H\subset \PP^6$ such that $D:= S\cap H$ is a smooth curve and then we have an exact sequence
\begin{equation}\label{eqc1}
0 \to \Ii _{S,\PP^6}(1) \to \Ii _{S,\PP^6}(2) \to \Ii _{D,H}(2)\to 0.
\end{equation}
The adjuntion formula gives $\omega _D \cong \Oo _D$ and so $D$ is a linearly normal elliptic curve of $H$. Therefore $D$ is projectively normal, $h^1(H,\Ii _{D,H}(2)) =0$ and $\Ii _{D,H}(2)$ is globally generated. Since $S$ is linearly normal, we have $h^1( \Ii _{S,\PP^6}(1))=0$ and so $\Ii _{S,\PP^6}(2)$ is spanned at each point of $H$ from (\ref{eqc1}).

Let us fix a point $P\in \PP^6$ and let $H$ be a general hyperplane with $P\in H$. To show that $\Ii _{S,\PP^6}(2)$ is spanned at $P$, it is sufficient to find $H$ so that $S\cap H$ is smooth. By the Bertini theorem the curve $S\cap H$ is smooth, except at most at $P$. If $P\notin S$, then $S\cap H$ is smooth. If $P\in S$, then we can take as $H$ a hyperplane not containing the tangent plane of $S$ at $P$.
\end{proof}

\begin{proposition}
On some $X_{2,2,3}\subset \PP^6$ there exists a globally generated vector bundle $\Ee$ of rank $2$ with the Chern classes $(c_1, c_2)=(2,18)$ and $H^0(\Ee(-1))=0$.
\end{proposition}
\begin{proof}
Let $S\subset \PP^6$ be a weak del Pezzo surface of degree $6$, possibly with only finitely many singular points (see \cite{d}). It is projectively normal and cut out by quadrics by Lemma \ref{d1}. Therefore there are two quadrics $U$ and $U'$ containing $S$ such that $\dim (U\cap U') =4$ and $U\cap U'$ is smooth outside $S$. Let $W\subset \PP^6$ be a cubic hypersurface such that $C:= W\cap S$, a scheme-theoretic intersection, is a smooth curve and $X:= U\cap U' \cap W$ is smooth outside $S$, e.g. we can take as $W$ a general cubic hypersurface by the Bertini theorem. Then the adjunction formula gives $\omega _C\cong \Oo _C(2)$. Since $S$ is scheme-theoretically cut out by quadrics, so $\Ii _{C,X}(2)$ is spanned. Note that $X=X_{2,2,3}$ is a CICY of degree $12$.
\end{proof}


\bibliographystyle{amsplain}

\begin{thebibliography}{10}

\bibitem{a}
Bj{\o}rn {\AA}dlandsvik, \emph{Joins and higher secant varieties}, Math. Scand.
  \textbf{61} (1987), no.~2, 213--222. \MR{947474 (89j:14030)}

\bibitem{am}
Cristian Anghel and Nicolae Manolache, \emph{Globally generated vector bundles
  on {$\Bbb P^n$} with {$c_1=3$}}, Math. Nachr. \textbf{286} (2013), no.~14-15,
  1407--1423. \MR{3119690}

\bibitem{Arrondo}
Enrique Arrondo, \emph{A home-made {H}artshorne-{S}erre correspondence}, Rev.
  Mat. Complut. \textbf{20} (2007), no.~2, 423--443. \MR{2351117 (2008g:14084)}

\bibitem{bc}
E.~Ballico and A.~Cossidente, \emph{Surfaces in {${\bf P}^5$} which do not
  admit trisecants}, Rocky Mountain J. Math. \textbf{29} (1999), no.~1, 77--91.
  \MR{1687656 (2000f:14081)}

\bibitem{BHM}
E.~Ballico, S.~Huh, and F.~Malaspina, \emph{Globally generated vector bundles
  of rank 2 on a smooth quadric threefold}, J. Pure Appl. Algebra \textbf{218}
  (2014), no.~2, 197--207. \MR{3120621}

\bibitem{BCKS}
V.~Batyrev, I.~Ciocan-Fontanine, B.~Kim and D.~van Straten, \emph{Conifold transitions and mirror symmetry for Calabi-Yau complete intersections in Grassmannians}, Nucl. Phys. \textbf{B} \textbf{514} (1998), no.~3, 640--666.


\bibitem{ccg}
Enrico Carlini, Luca Chiantini, and Anthony~V. Geramita, \emph{Complete
  intersections on general hypersurfaces}, Michigan Math. J. \textbf{57}
  (2008), 121--136, Special volume in honor of Melvin Hochster. \MR{2492444
  (2010b:14098)}

\bibitem{ckm}
Marc Coppens, Changho Keem, and Gerriet Martens, \emph{Primitive linear series
  on curves}, Manuscripta Math. \textbf{77} (1992), no.~2-3, 237--264.
  \MR{1188583 (93j:14028)}

\bibitem{dm}
Olivier Debarre and Laurent Manivel, \emph{Sur la vari\'et\'e des espaces
  lin\'eaires contenus dans une intersection compl\`ete}, Math. Ann.
  \textbf{312} (1998), no.~3, 549--574. \MR{1654757 (99j:14048)}

\bibitem{d}
Susan~J. Diesel, \emph{Irreducibility and dimension theorems for families of
  height {$3$} {G}orenstein algebras}, Pacific J. Math. \textbf{172} (1996),
  no.~2, 365--397. \MR{1386623 (99f:13016)}

\bibitem{egh}
David Eisenbud, Mark Green, and Joe Harris, \emph{Cayley-{B}acharach theorems
  and conjectures}, Bull. Amer. Math. Soc. (N.S.) \textbf{33} (1996), no.~3,
  295--324. \MR{1376653 (97a:14059)}

\bibitem{ES}
Geir Ellingsrud and Stein~Arild Str{\o}mme, \emph{Bott's formula and
  enumerative geometry}, J. Amer. Math. Soc. \textbf{9} (1996), no.~1,
  175--193. \MR{1317230 (96j:14039)}

\bibitem{ev}
H{\'e}l{\`e}ne Esnault and Eckart Viehweg, \emph{Lectures on vanishing
  theorems}, DMV Seminar, vol.~20, Birkh\"auser Verlag, Basel, 1992.
  \MR{1193913 (94a:14017)}

\bibitem{fov}
H.~Flenner, L.~O'Carroll, and W.~Vogel, \emph{Joins and intersections},
  Springer Monographs in Mathematics, Springer-Verlag, Berlin, 1999.
  \MR{1724388 (2001b:14010)}

\bibitem{f1}
Takao Fujita, \emph{Classification of projective varieties of {$\Delta $}-genus
  one}, Proc. Japan Acad. Ser. A Math. Sci. \textbf{58} (1982), no.~3,
  113--116. \MR{664549 (83g:14003)}

\bibitem{f2}
\bysame, \emph{Classification theories of polarized varieties}, London
  Mathematical Society Lecture Note Series, vol. 155, Cambridge University
  Press, Cambridge, 1990. \MR{1162108 (93e:14009)}

\bibitem{HK}
B.~Haghighat and A.~Klemm, \emph{Topological strings on Grassmannian Calabi-Yau manifolds}, JHEP \textbf{01} (2009), no.~029.

\bibitem{he}
Joe Harris, \emph{Curves in projective space}, S\'eminaire de Math\'ematiques
  Sup\'erieures [Seminar on Higher Mathematics], vol.~85, Presses de
  l'Universit\'e de Montr\'eal, Montreal, Que., 1982, With the collaboration of
  David Eisenbud. \MR{685427 (84g:14024)}

\bibitem{h0}
Robin Hartshorne, \emph{Algebraic geometry}, Springer-Verlag, New York, 1977,
  Graduate Texts in Mathematics, No. 52. \MR{MR0463157 (57 \#3116)}

\bibitem{hs}
\bysame, \emph{Stable vector bundles of rank {$2$} on {${\bf P}^{3}$}}, Math.
  Ann. \textbf{238} (1978), no.~3, 229--280. \MR{514430 (80c:14011)}

\bibitem{ho}
Andrea Hofmann, \emph{The degree of the third secant variety of a smooth curve
  of genus $2$}, preprint, arXiv:1103.4655 [math.AG], 2011.

\bibitem{kanazawa}
Atsushi Kanazawa, \emph{Pfaffian {C}alabi-{Y}au threefolds and mirror
  symmetry}, Commun. Number Theory Phys. \textbf{6} (2012), no.~3, 661--696.
  \MR{3021322}

\bibitem{Knutsen}
Andreas~Leopold Knutsen, \emph{On isolated smooth curves of low genera in
  {C}alabi-{Y}au complete intersection threefolds}, Trans. Amer. Math. Soc.
  \textbf{364} (2012), no.~10, 5243--5264. \MR{2931328}

\bibitem{madonna}
Carlo~G. Madonna, \emph{A{CM} vector bundles on prime {F}ano threefolds and
  complete intersection {C}alabi-{Y}au threefolds}, Rev. Roumaine Math. Pures
  Appl. \textbf{47} (2002), no.~2, 211--222 (2003). \MR{1979043 (2004g:14048)}

\bibitem{man}
N.~Manolache, \emph{Globally generated vector bundles on $\mathbb{P}^3$ with
  $c_1=3$}, Preprint, arXiv:1202.5988 [math.AG], 2012.

\bibitem{sierra}
Jos{\'e}~Carlos Sierra, \emph{A degree bound for globally generated vector
  bundles}, Math. Z. \textbf{262} (2009), no.~3, 517--525. \MR{2506304
  (2010i:14076)}

\bibitem{SU}
Jos{\'e}~Carlos Sierra and Luca Ugaglia, \emph{On globally generated vector
  bundles on projective spaces}, J. Pure Appl. Algebra \textbf{213} (2009),
  no.~11, 2141--2146. \MR{2533312 (2010d:14062)}

\bibitem{thomas}
R.P.~Thomas, \emph{A holomorphic Casson invariant for Calabi-Yau $3$-folds, and bundles on $K3$ fibrations}, J. Differential Geom. \textbf{54} (2000), no.~2, 367--438. \MR{1818182 (2002b:14049)}


\end{thebibliography}
\providecommand{\bysame}{\leavevmode\hbox to3em{\hrulefill}\thinspace}
\providecommand{\MR}{\relax\ifhmode\unskip\space\fi MR }
\providecommand{\MRhref}[2]{%
  \href{http://www.ams.org/mathscinet-getitem?mr=#1}{#2}
}
\providecommand{\href}[2]{#2}

\end{document}